\documentclass[11pt,leqno]{article}
\usepackage{graphicx, amsfonts, amsthm, amsxtra, amssymb, verbatim,latexsym,bbm}
\usepackage{calrsfs}
\usepackage{addfont}
\usepackage[mathscr]{euscript}
\usepackage[all,cmtip]{xy}
\usepackage{xcolor}

\begin{document}

\def\Z{\mathbb{Z}}                   
\def\Q{\mathbb{Q}}                   
\def\C{\mathbb{C}}                   
\def\N{\mathbb{N}}                   
\def\dR{{\rm dR}}                    
\def\Gm{\mathbb{G}_m}                 
\def\Ga{\mathbb{G}_a}                 
\def\Tr{{\rm Tr}}                      
\def\tr{{{\mathsf t}{\mathsf r}}}                 
\def\spec{{\rm Spec}}            
\def\ker{{\rm ker}}              
\def\GL{{\rm GL}}                
\def\k{{\sf k}}                     
\def\ring{{ R}}                   
\def\X{{\sf X}}                      
\def\T{{\sf T}}                      
\def\Ts{{\sf S}}
\def\cmv{{\sf M}}                    
\def\BG{{\sf G}}                       
\def\podu{{\sf pd}}                   
\def\ped{{\sf U}}                    
\def\per{{\sf  P}}                   
\def\gm{{\sf  A}}                    
\def\gma{{\sf  B}}                   
\def\ben{{\sf b}}                    

\def\Rav{{\mathfrak M }}                     
\def\Ram{{\mathfrak C}}                     
\def\Rap{{\mathfrak G}}                     

\def\nov{{\sf  n}}                    
\def\mov{{\sf  m}}                    
\def\Yuk{{\sf Y}}                     
\def\Ra{{\sf R}}                      
\def\hn{{ h}}                      
\def\cpe{{\sf C}}                     
\def\g{{\sf g}}                       
\def\t{{\sf t}}                       
\def\pedo{{\sf  \Pi}}                  

\def\Der{{\rm Der}}                   
\def\MMF{{\sf MF}}                    
\def\codim{{\rm codim}}                
\def\dim{{\rm    dim}}                
\def\Lie{{\rm Lie}}                   
\def\gg{{\mathfrak g}}                

\def\u{{\sf u}}                       

\def\imh{{  \Psi}}                 
\def\imc{{  \Phi }}                  
\def\stab{{\rm Stab }}               
\def\Vec{{\rm Vec}}                 
\def\prim{{\rm prim}}                  

\def\Fg{{\sf F}}     
\def\hol{{\rm hol}}  
\def\non{{\rm non}}  
\def\alg{{\rm alg}}  

\def\bcov{{\rm \O_\T}}       

\def\leaves{{\cal L}}        

\def\GM{{\rm GM}}

\def\perr{{\sf q}}        
\def\perdo{{\cal K}}   
\def\sfl{{\mathrm F}} 
\def\sp{{\mathbb S}}  

\newcommand\diff[1]{\frac{d #1}{dz}} 
\def\End{{\rm End}}              

\def\sing{{\rm Sing}}            
\def\cha{{\rm char}}             
\def\Gal{{\rm Gal}}              
\def\jacob{{\rm jacob}}          
\def\tjurina{{\rm tjurina}}      
\newcommand\Pn[1]{\mathbb{P}^{#1}}   
\def\Ff{\mathbb{F}}                  

\def\O{{\cal O}}                     
\def\as{\mathbb{U}}                  
\def\ring{{ R}}                         
\def\R{\mathbb{R}}                   

\def\Mat{{\rm Mat}}              
\def\cl{{\rm cl}}                

\def\hc{{\mathsf H}}                 
\def\Hb{{\cal H}}                    
\def\pese{{\sf P}}                  

\def\PP{\tilde{\cal P}}              
\def\K{{\mathbb K}}                  

\def\M{{\cal M}}
\def\RR{{\cal R}}
\newcommand\Hi[1]{\mathbb{P}^{#1}_\infty}
\def\pt{\mathbb{C}[t]}               
\def\W{{\cal W}}                     
\def\gr{{\rm Gr}}                
\def\Im{{\rm Im}}                
\def\Re{{\rm Re}}                
\def\depth{{\rm depth}}
\newcommand\SL[2]{{\rm SL}(#1, #2)}    
\newcommand\PSL[2]{{\rm PSL}(#1, #2)}  
\def\Resi{{\rm Resi}}              

\def\L{{\cal L}}                     
\def\Aut{{\rm Aut}}              
\def\any{R}                          
\newcommand\ovl[1]{\overline{#1}}    

\newcommand\mf[2]{{M}^{#1}_{#2}}     
\newcommand\mfn[2]{{\tilde M}^{#1}_{#2}}     

\newcommand\bn[2]{\binom{#1}{#2}}    
\def\ja{{\rm j}}                 
\def\Sc{\mathsf{S}}                  
\newcommand\es[1]{g_{#1}}            
\newcommand\V{{\mathsf V}}           
\newcommand\WW{{\mathsf W}}          
\newcommand\Ss{{\cal O}}             
\def\rank{{\rm rank}}                
\def\Dif{{\cal D}}                   
\def\gcd{{\rm gcd}}                  
\def\zedi{{\rm ZD}}                  
\def\BM{{\mathsf H}}                 
\def\plf{{\sf pl}}                             
\def\sgn{{\rm sgn}}                      
\def\diag{{\rm diag}}                   
\def\hodge{{\rm Hodge}}
\def\HF{{\sf F}}                                
\def\WF{{\sf W}}                               
\def\HV{{\sf HV}}                                
\def\pol{{\rm pole}}                               
\def\bafi{{\sf r}}
\def\id{{\rm id}}                               
\def\gms{{\sf M}}                           
\def\Iso{{\rm Iso}}                           

\def\hl{{\rm L}}    
\def\imF{{\rm F}}
\def\imG{{\rm G}}
\def\cy{{CY }}
\def\rc{{Rankin-Cohen }}
\def\rs{{Ramanujan-Serre type derivation }}
\def\DHR{{\rm DHR }}            
\def\H{{\sf{ E}}}
\def\P {\mathbb{P}}                  
\def\E{{\cal E}}
\def\L{{\sf L}}             
\def\CX{{\cal X}}
\def\dt{{\sf d}}             
\def\LG{{\sf G}}   
\def\LA{{\rm Lie}(\LG)}   
\def\amsy{\mathfrak{G}}  
\def\gG{{\sf g}}   
\def\gL{\mathfrak{g}}   
\def\rvf{{\sf H}}   
\def\cvf{{\sf F}}   
\def\di2{{ m}}   
\def\a{\mathfrak{a}}
\def\b{\mathfrak{b}}

\def\mmat{e}
\def\cmat{f}
\def\rmat{h}

\def\hn{{\sf h}}                      

\def\tgtwo{\mathfrak{t}}
\def\tgone{\mathfrak{t}}
\def\tgzero{\mathfrak{t}}
\def\kgone{\mathfrak{k}}
\def\ggtwo{\mathfrak{g}}
\def\ggzero{\mathfrak{g}}

\def\Yukk{{\sf C}^\alg}

\def\sl2{\mathfrak{sl}_2(\C)}
\def\SL2{{\rm SL}_2(\Z)}
\def\qmfs{\widetilde{\mathscr{M}}}             
\def\mfs{{\mathscr{M}}}                            
\def\cyqmfs{{ \widetilde{\mathcal{M}}}}             
\def\cymfs{{\mathcal{M}}}                            
\def\Ramvf{{\sf Ra}}             
\def\rcv{{\sf D}}             
\def\rcdo{{\mathcal{D}}}             
\def\rcdomf{{\mathscr{D}}}
\def\mdo{{\mathcal{R}}}             
\def\rdo{{\mathcal{H}}}             
\def\cdo{{\mathcal{F}}}             
\def\rsdo{{\partial}}             
\newtheorem{theo}{Theorem}[section]
\newtheorem{exam}{Example}[section]
\newtheorem{coro}{Corollary}[section]
\newtheorem{defi}{Definition}[section]
\newtheorem{prob}{Problem}[section]
\newtheorem{lemm}{Lemma}[section]
\newtheorem{prop}{Proposition}[section]
\newtheorem{rem}{Remark}[section]
\newtheorem{obs}{Observation}[section]
\newtheorem{conj}{Conjecture}
\newtheorem{nota}{Notation}[section]
\newtheorem{ass}{Assumption}[section]
\newtheorem{calc}{}
\numberwithin{equation}{section}

\begin{center}
{\LARGE\bf Rankin-Cohen brackets for Calabi-Yau modular forms}
\\

\vspace{.25in} {\large {\sc
Younes Nikdelan \footnote{Departamento de An\'alise Matem\'atica, Instituto de Matem\'atica
e Estat\'{i}stica (IME), Universidade do Estado do
Rio de Janeiro (UERJ), Rua S\~{a}o Francisco Xavier, 524, Rio de Janeiro, Brazil / CEP: 20550-900. e-mail: younes.nikdelan@ime.uerj.br}}} \\
\end{center}
\vspace{.25in}
\begin{abstract}
For any positive integer $n$, we introduce a quasi-homogeneous vector field $\textsf{D}$ of degree $2$
on a moduli space $\textsf{T}$ of enhanced Calabi-Yau $n$-folds arising from the Dwork family.
By Calabi-Yau quasi-modular forms for Dwork family
we mean the elements of the graded $\mathbb{C}$-algebra $\widetilde{\mathcal{M}}$ generated by the components of
a particular solution of $\textsf{D}$, which are provided with natural weight. Using
$\textsf{D}$ we introduce the derivation $\mathcal{D}$ and the Ramanujan-Serre type
derivation $\partial$ on $\widetilde{\mathcal{M}}$. We show that they are
degree $2$ differential operators and there exists a proper subspace $\mathcal{M}\subset \widetilde{\mathcal{M}}$,
called the space of Calabi-Yau modular forms, which is closed under $\partial$. Using the
derivation $\mathcal{D}$, we define the Rankin-Cohen brackets for Calabi-Yau quasi-modular
forms and prove that the subspace generated by the positive weight elements of
$\mathcal{M}$ is closed under the Rankin-Cohen brackets.
\end{abstract}


\section{Introduction}\label{section introduction}
The proof of Fermat's last theorem led to the celebrated modularity theorem, which states that
elliptic curves over the field of rational numbers $\Q$ are related with modular forms. Elliptic curves are $1$-dimensional Calabi-Yau (\cy) varieties, which makes it natural to ask whether a similar statement of modularity holds for higher dimensional \cy varieties. This question persuaded mathematicians and theoretical physicists to the subject of \emph{ modularity of \cy manifolds} which is one of the considerable present challenges of the modern algebraic number theory. Some relevant results can be found, for instance, in \cite{ny13} and the references therein. Yui in \cite{ny13} divides  the modularity of \cy varieties in arithmetic modularity and geometric modularity including  (1) the modularity (automorphy) of Galois representations of \cy varieties (or motives) defined over $\Q$ or number fields, (2) the modularity of solutions of Picard-Fuchs differential equations of families of \cy varieties, and mirror maps (mirror moonshine), (3) the modularity of generating functions of invariants counting certain quantities on \cy varieties, and (4) the modularity of moduli for families of \cy varieties. But so far, in a general context, even there is no unified formulation or statement of the modularity of \cy varieties. Yamaguchi and Yau \cite{yaya} in 2004 showed that the partition functions for the mirror quintic can be expressed in terms of finitely many generators of a differential ring, which somehow play the role of quasi-modular forms; then Alim and Lange \cite{alla} in 2007 generalized their results for arbitrary CY $3$-folds.
Movasati in \cite{GMCD-MQCY3} says: "All the attempts to find an arithmetic modularity for mirror quintic have failed, and this might be an indication that maybe such varieties need a new kind of modular forms." In this way, he introduced \cy (quasi-)modular forms which somehow can be considered as a modern generalization of the classical quasi-modular forms (automorphic forms) theory. The present paper provides some evidences in favor of this generalization; namely, we introduce the space of \cy quasi-modular forms $\cyqmfs$ for the Dwork family and furnish it with a Rankin-Cohen algebra structure. Then we find a proper subspace of $\cyqmfs$ which is closed under the Rankin-Cohen brackets. This can be considered as a generalization of the work of Zagier \cite{zag94} for the space of classical (quasi-)modular forms.

Movasati in \cite{ho14} used an algebraic method, called \emph{Gauss-Manin connection
in disguise} (GMCD), in a geometric framework and reencountered the Ramanujan \cite{ra16} vector field  (system) $\Ramvf$ (see \eqref{eq ramanujan}) on certain moduli of a family of enhanced elliptic curves (see \eqref{eq gmcd of Ra} and \eqref{eq ufec}).
It is known that the triple of Eisenstein series $(E_2,E_4,E_6)$
gives a solution of the Ramanujan system $\Ramvf$, and the space of modular forms $\mathscr{M}$ and quasi-modular forms $\widetilde{\mathscr{M}}$ for $\SL2$ are respectively graded $\C$-algebras $\mathscr{M}=\C[E_4,E_6]$ and $\widetilde{\mathscr{M}}=\C[E_2,E_4,E_6]$.
Note that $E_4,\ E_6$ are modular forms of weight $4$ and $6$, respectively, and $E_2$ is a quasi-modular form of weight $2$ which is not modular. After this work, in the paper \cite{ho22} he applied GMCD to the family of mirror quintic $3$-fold and a few years later
expanded it to the book \cite{GMCD-MQCY3}, where he introduced CY modular forms for the mirror quintic $3$-fold. In particular, he reencountered the so-called \emph{Yukawa coupling} of Candelas et al. \cite{can91} and expressed it in terms of CY modular forms, and also by considering \cite{bcov93} and \cite{yaya}, he wrote the  topological string partition functions for the mirror quintic $3$-fold in terms of CY modular forms. The mirror quintic $3$-fold is the particular case $n=3$ of families of mirror $n$-folds, $n\in \Z_{>0}$, arising from the so-called Dwork family (see \cite{grmorpl}). The author and  Movasati in \cite{movnik}  applied GMCD to the families of the mirror $n$-folds arising from the Dwork family, for all positive integers $n$, which briefly is as follows.
We considered the moduli space $\T=\T_n$ of the pairs $(X,[\alpha_1,\alpha_2,\ldots,\alpha_n,\alpha_{n+1}])$, where
$X$ is a mirror $n$-fold arising from the Dwork
family and $\{\alpha_1,\alpha_2,\ldots,\alpha_{n+1}\}$
refers to a basis of the $n$-th algebraic de Rham cohomology $H_\dR^n(X)$ which is compatible with the
Hodge filtration of $H_\dR^n(X)$ (see \eqref{eq bcwhf}) and its intersection form matrix is constant (see \eqref{11jan2016}). We showed that there exist a unique vector field $\Ra=\Ra_n$, called \emph{modular vector field}, and regular functions $\Yuk_i, \ 1\leq i \leq n-2$, that satisfy certain equation involving the Gauss-Manin connection of the universal family of $\T$ (see Theorem \ref{main3} and also
\cite[Theorem 1.1]{younes1} in a more general context). Due to \cite{ho14} we can say that the modular vector field $\Ra$ is a generalization of the Ramanujan vector field $\Ramvf$. For $n=1,2,3,4$ we found the $q$-expansion of solution components of the modular vector field $\Ra$ whose coefficients are  surprisingly integers. Actually, for $n=1,2$, where $\T$ is the moduli of enhanced elliptic curves and K3-surfaces, respectively, the solution components, as it was expected, are quasi-modular forms (see \eqref{lovely-1} and \eqref{eq solution R01}). See also \cite{ali17} for similar computations. In the case $n=3$, $\Ra_3$ is explicitly computed in \cite{ho22} and it is
verified that $\Yuk_1$ is the Yukawa coupling introduced in
\cite{can91}, which predicts the numbers of rational curves of various degrees on a general quintic three-fold.
For $n=4$, we computed the modular vector field $\Ra_4$ explicitly in \cite{movnik} and we observed that
$\Yuk_1^2=\Yuk_2^2$ is the same as $4$-point function presented in \cite[Table~1, $d=4$]{grmorpl}, and we computed the modular coordinate $z$ given in \cite[\S 6.1]{KlemmPandharipande2008} in terms of solution components of $\Ra_4$.
Unlike the cases $n=1,2$, for $n=3,4$ we believe that it is not possible to write the solution components of $\Ra$ in terms of classical quasi-modular forms, since the coefficients of their $q$-expansions increase very rapidly. This leads us to think to another theory which generalizes the theory of quasi-modular forms, where the space generated by solution components of $\Ra$ is the adequate  candidate of the desired generalization.

One of the initial steps in the above-mentioned generalization is the correct assignment of weights  to the components of a solution of $\Ra$. In order to do this, we recall another important property of the Ramanujan vector field $\Ramvf$.
We can easily observe that the Lie algebra generated by the Ramanujan vector field $\Ramvf=\frac{1}{12}(t_1^2-t_2)\frac{\partial}{\partial t_1}+\frac{1}{3}(t_1t_2-t_3)\frac{\partial}{\partial t_2}+\frac{1}{2}(t_1t_3-t_2^2)\frac{\partial}{\partial t_3}$, the radial vector field $H=2t_1\frac{\partial}{\partial t_1}+4t_2\frac{\partial}{\partial t_2}+6t_3\frac{\partial}{\partial t_3}$ and the constant vector field $F=-12\frac{\partial}{\partial t_1}$ is isomorphic to the Lie algebra $\sl2$ (remember that $(t_1,t_2,t_3)=(E_2,E_4,E_6)$ is a solution of $\Ramvf$). Note that $\deg(E_2)=2,\ \deg(E_4)=4, \ \deg(E_6)=6$ and these integers appear as coefficients of the components of the vector field $H$. Moreover, we have $F(t_1)=-12, \ F(t_2)=F(t_3)=0$; indeed, if we consider $F$ as a derivation on $\qmfs$, then  $\ker F= \mfs $.
Our attention in \cite{younes2} was dedicated to the Lie algebra $\sl2$ and we proved that for any $n$, there are vector fields $\rvf$ and $\cvf$ on $\T=\T_n$ such that along with the modular vector field $\Ra$ generate a copy of $\sl2$ in $\mathfrak{X}(\T)$ (see Theorem \ref{theo 3})(the notations $\rvf$ and $\cvf$ in the whole manuscript are used for the same vector fields given in Theorem \ref{theo 3}). Furthermore,  we observe that the vector field $\rvf$ can be written in the form $\rvf=\sum_{j=1}^{\dt}w_jt_j\frac{\partial }{\partial t_j}$, where $\dt=\dim \T$, $(t_1,t_2,\ldots,t_\dt)$ is a chart of $\T$, which will be constructed in Subsection \ref{subsection EMS}, and $w_j\in \Z_{\geq 0}, \ j=1,2,\ldots,\dt$ (see \eqref{eq rvf gf}).
These facts lead us to define $\deg(t_j):=w_j,\ j=1,2,\ldots,\dt$. By applying these weights, in Proposition \ref{prop Risqh2} we show that for any positive integer $n$ the modular vector field $\Ra=\Ra_n$ is a quasi-homogeneous vector field of degree $2$.
If $n=1$ or $n$ is even, then $\cvf(t_j)=0$, for all $j\neq 2$, and $\cvf(t_2)\neq0$. But, if $n\geq 3$ is odd, then we observe that  $\cvf(t_j)=0$, for all $j\neq 2,\dt$, and $\cvf(t_2)\neq 0$, $\cvf(t_\dt)\neq 0$ which will cause problems for our purposes in Section \ref{section RCACY}. To avoid these problems, we introduce another vector field $\rcv$ (see Subsection \ref{subsection rcv}), which coincides with $\Ra$ when $n=1$ or $n$ is even, but for odd $n\geq 3$ it can be different from $\Ra$. We prove that $\rcv$ along with $\rvf$ and the constant vector field $(1+\delta_2^n)\frac{\partial}{\partial t_2}$ forms a copy of $\sl2$, where $\delta_2^n$ is the Kronecker delta, and $\rcv$ is a quasi-homogeneous vector field of degree $2$ in $\T$ (see Lemma \ref{lemm fund} and Corollary \ref{cor sl2c D}).
Now, suppose that $\t_j, \ j=1,2,\ldots,\dt$, is the component of a particular solution of $\rcv$ associated with the coordinate chart $t_j$ carrying the same weight, i.e., $\deg(\t_j)=w_j$. We define the space of \emph{\cy quasi-modular forms} for Dwork family as $\cyqmfs :=\C[\t_1,\t_2,\t_3,\ldots,\t_{\dt},\frac{1}{\t_{n+2}(\t_{n+2}-\t_1^{n+2})\check{\t}}]$ and the space of \emph{\cy modular forms} for Dwork family as $\cymfs :=\C[\t_1,\widehat{\ \t_2},\t_3,\t_4,\ldots,\t_{\dt},\frac{1}{\t_{n+2}(\t_{n+2}-\t_1^{n+2})\check{\t}}]$, where $\check{\t}$ is a product of a few number of  $\t_j$'s (see \eqref{eq rfs}) and the symbol $\widehat{\ \t_2}$ means that the component $\t_2$ is omitted, i.e.,  $\t_2\notin \cymfs$; indeed $\cymfs$ is a subspace of $\cyqmfs$, and $\cyqmfs=\cymfs[\t_2]$. For any $n$, Remark \ref{rem weights} yields that $\deg (\t_2)=2$.  In our approach $\t_2$ plays the same role of the quasi-modular form $E_2$ in the theory of quasi-modular forms for $\SL2$, which gives sense to the definition of $\cymfs$ (recall that $\qmfs=\mfs[E_2]$).
Throughout by \cy quasi-modular forms or \cy modular forms we mean \cy quasi-modular forms or \cy modular forms for the Dwork family.

To motivate and explain better our main results, we recall again some known facts of classical theory of quasi-modular forms. It is well known that the derivative of a modular form is not necessarily a modular form. More precisely, for any positive integer $r$ and any modular form $f\in \mathscr{M}_r$ of weight $r$ for $\SL2$, we know that $f'\in \widetilde{\mathscr{M}}_{r+2}$ is a quasi-modular form of weight $r+2$ which is not necessarily modular.
But the derivative $f'$ can be corrected using the Ramanujan-Serre derivation $\partial f=f'-\frac{1}{12}rE_2f$ which yields $\partial f \in \mathscr{M}_{r+2}$ (see \eqref{eq qmf der} and \eqref{eq rsd2}). Rankin in \cite{ran56} described some necessary conditions under which a polynomial in a given modular form and its derivatives is again a modular form. Cohen \cite{coh77} generalized the result of Rankin and for any non-negative integer $k$, defined a bilinear operator $F_k(\cdot,\cdot)$ and proved that for all $f\in\mathscr{M}_r,\ g\in\mathscr{M}_s$ one gets $F_k(f,g)\in \mathscr{M}_{r+s+2k}$. Later, Zagier in \cite{zag94} called these bilinear forms as \emph{Rankin-Cohen brackets} and denoted them by $[\cdot,\cdot]_k$ (see \eqref{eq rcb}). Furthermore, he developed the theory of  \emph{Rankin-Cohen algebras}, which are briefly described  in Section \ref{section RCa}. The principal objective of this paper is to endow $\cyqmfs$ and $\cymfs$ with standard Rankin-Cohen and canonical Rankin-Cohen algebra structure, respectively. In order to do this we will need a degree $2$ differential operator and a Ramanujan-Serre type derivation on $\cyqmfs$ and $\cymfs$, respectively.
To this end, we observe that $\rcv$ induces a differential operator on $\cyqmfs$ which is denoted by $\rcdo$ (see \eqref{eq rcdo}). It is not difficult to observe that the space of \cy modular forms $\cymfs$ is not closed under $\rcdo$, but by
correcting the derivation $\rcdo$ we can define the \emph{Ramanujan-Serre type derivation}
$\rsdo$ (see \eqref{eq rsdo}). In the following theorem we state the first main result of this work.

\begin{theo} \label{theo 4}
Let $\rcdo$ and $\rsdo$ be the derivations defined in \eqref{eq rcdo} and \eqref{eq rsdo}, respectively. Then the following hold.
\begin{enumerate}
  \item The derivation $\rcdo$ is a degree $2$ differential operator on $\cyqmfs$.
  \item The \rs $\rsdo$ is a degree $2$ differential operator on $\cymfs$.
\end{enumerate}
\end{theo}

We emphasize that, due to Theorem \ref{theo 4},  $\cymfs$ is closed under $\rsdo$, and in particular for all integers $r$ we have $\rsdo:\cymfs_r\to\cymfs_{r+2}$.
Using the derivation $\rcdo$, for any non-negative integers $k,s,r$ and any $f\in \cyqmfs_r,  g\in \cyqmfs_s$, we define the $k$-th Rankin-Cohen bracket $[f,g]_{\rcdo,k}$
of \cy quasi-modular forms in \eqref{eq rcb r1} and observe that $[f,g]_{\rcdo,k}\in \cyqmfs_{r+s+2}$. Indeed, $[\cdot,\cdot]_{\rcdo,k}$ provides $\cyqmfs$ with a standard Rankin-Cohen algebra structure.
Finally, in the following theorem we establish the second main, and more important, result of the present paper.

\begin{theo} \label{theo 5}
For all non-negative integers $r,s,k$ and for any $f\in \cymfs_r$, $g\in \cymfs_s$  we have:
\[
[f,g]_{\rcdo,k}\in \cymfs_{r+s+2k}\,.
\]
\end{theo}

In the other words, Theorem \ref{theo 5} says that the space of \cy modular forms of positive weight is closed under
the Rankin-Cohen brackets of the \cy quasi-modular forms, and hence we provide this space with a canonical Rankin-Cohen algebra structure. We prove  Theorem \ref{theo 4}
and Theorem \ref{theo 5} in Section \ref{section RCACY}.
It is worth to mention that for various  examples
of \cy modular forms of negative weight we used the computer and observed that their Rankin-Cohen brackets are again \cy modular forms. Thus, we conjecture
that the whole space of the \cy modular forms $\cymfs$ is closed under the Rankin-Cohen brackets.

This manuscript is organized as follows. In Section \ref{section RCa} we briefly review the relevant definitions and facts of \cite{zag94}
which will be used in the rest of the text. Section \ref{section MSGMCD} starts with a short summary of \cite{movnik} and \cite{younes2} in Subsections \ref{subsection EMS} and \ref{subsection amsy} which constructs the foundation
of the present research and also lets us to have a self contained manuscript. After that, in Subsection \ref{subsection R as qhvf} we prove that the modular vector field $\Ra$ is a quasi-homogeneous vector field
of degree $2$. We introduce the vector field $\rcv$ in Section \ref{subsection rcv} and demonstrate the fundamental lemma.
In Section \ref{section RCACY} our main results are stated and proved. Namely, we define the concepts of: spaces of \cy quasi-modular forms and \cy modular forms, derivation
$\rcdo$, \rs $\rsdo$ and Rankin-Cohen brackets of the \cy quasi-modular forms. We provide the proofs of Theorem \ref{theo 4} and Theorem \ref{theo 5} in this section. In various examples
of the same section, for $n=1,2,3,4$,  the derivations $\rcdo,\, \rsdo$ and Rankin-Cohen brackets of a few \cy modular forms are explicitly calculated. Section \ref{section FR}
deals with the final remarks. In this section we state a conjecture which improves our results. \\

{\bf Acknowledgment.} The initial inspiration of the present study came from a conversation between Hossein Movasati and Don Zagier, which was shared later with the author and others by Movasati. At that moment we did not succeed in solving the problem, because of the absence of some key points such as the correct weight of the \cy (quasi-)modular forms and etc. After the work \cite{younes2}, the author could find the missing points of the research and completed the present work. Because of this, the author would like to thank both Movasati and Zagier, in particular he is very grateful to Movasati for his helpful discussions and comments.

\section{Rankin-Cohen algebra} \label{section RCa}
In this section we recall the important facts and terminologies of \cite{zag94} which are necessary for the present paper.
Let $\widetilde{\mathscr{M}}=\bigoplus_{r\geq 0}\qmfs_r$ and $\mathscr{M}=\bigoplus_{r\geq 0}\mathscr{M}_r$, respectively, be the graded algebras of quasi-modular forms and of modular forms, where $\qmfs_r:=\qmfs_r(\SL2)$ and $\mathscr{M}_r:=\mathscr{M}_r(\SL2)$, respectively, are the spaces of quasi-modular forms and of modular forms of weight $r$ for $\SL2$. It is well known  that $\widetilde{\mathscr{M}}=\C[E_2,E_4,E_6]$ and  $\mathscr{M}=\C[E_4,E_6]$, where  $E_2, E_4, E_6$ are Eisenstein series given as:
\begin{align}
&E_{2j}(q)=1+b_j\sum_{k=1}^\infty \sigma_{2j-1}(k) q^{k}\,\, \textrm{with}\,\,
\sigma_i(k)=\sum_{d\mid k}d^i, \ (b_1,b_2,b_3)=(-24,240,-504)\,.\label{eq eis ser}
\end{align}
Note that $E_4$ and $E_6$ are modular forms of weight 4 and 6, respectively, while $E_2$ is a quasi-modular form of weight 2 which is not modular. The triple $(E_2,E_4,E_6)$ satisfies the system of ordinary differential equations
\begin{equation} \label{eq ramanujan}
 {\Ramvf}:\left \{ \begin{array}{l}
t_1'=\frac{1}{12}(t_1^2-t_2) \\\\
t_2'=\frac{1}{3}(t_1t_2-t_3) \\\\
t_3'=\frac{1}{2}(t_1t_3-t_2^2)
\end{array} \right.,
\end{equation}
which is known as the \emph{Ramanujan relations between Eisenstein series}, and from now on we call it the \emph{Ramanujan vector field}.
Note that here $t_j'=q\frac{\partial t_j}{\partial q}=\frac{1}{2\pi i}\frac{d t_j}{d \tau}$ where $q=e^{2\pi i\tau}$ and $\tau\in \mathbb{H}:=\{z\in \C \ | \ {\rm Im}(z)>0\}.$
The Ramanujan vector field $\Ramvf=t'_1\frac{\partial }{\partial t_1}+t'_2\frac{\partial }{\partial t_2}+t'_3\frac{\partial }{\partial t_3}$ together with two vector fields $H=2t_1\frac{\partial }{\partial t_1}+4t_2\frac{\partial }{\partial t_2}+6t_3\frac{\partial }{\partial t_3}$ and
$F=-12\frac{\partial }{\partial t_1}$
forms a copy of $\sl2$; this follows from the fact that $[\Ramvf,F]=H \ , \ \ [H,\Ramvf]=2\Ramvf \ , \ \ [H,\cvf]=-2F$, where $[ \ ,\ ]$ refers to the Lie bracket of vector fields. We know that if $f\in \mathscr{M}_r$ is a modular form of weight $r$, then $f'$ is not necessarily a modular form. If instead of the usual derivation, we use the so-called \emph{Ramanujan-Serre derivation} $\partial$ given by
\begin{equation}\label{eq rsd}
  \rsdo f=f'-\frac{1}{12}rE_2f,
\end{equation}
then $\partial f$ is a modular form of weight $r+2$. After substituting  $(t_1,t_2,t_3)$  by $(E_2,E_4,E_6)$ in the Ramanujan vector field \eqref{eq ramanujan}, for any non-negative integer $r$ and any $f\in \qmfs_r$, we get $f'=\rcdomf f$ where the differential operator $\rcdomf$ on $\widetilde{\mathscr{M}}=\C[E_2,E_4,E_6]$ is given as follows:
\begin{equation}\label{eq qmf der}
  \rcdomf:\widetilde{\mathscr{M}}_r\to \widetilde{\mathscr{M}}_{r+2} \, ; \quad f'=\rcdomf f=\frac{E_2^2-E_4}{12}\frac{\partial f}{\partial E_2}+\frac{E_2E_4-E_6}{3}\frac{\partial f }{\partial E_4}+\frac{E_2E_6-E_4^2}{2}\frac{\partial f}{\partial E_6}\, ,
\end{equation}
which is a degree $2$ differential operator. Therefore, for any $f\in \mathscr{M}_r$ since $\frac{\partial f}{\partial E_2}=0$, we can express the Ramanujan-Serre derivation \eqref{eq rsd} as follows:
 \begin{equation}\label{eq rsd2}
  \rsdo f=-\frac{E_6}{3}\frac{\partial f}{\partial E_4}-\frac{E_4^2}{2}\frac{\partial f}{\partial E_6}\ ,
\end{equation}
from which we get that the Ramanujan-Serre derivation $\partial$ kills the terms which include $E_2$.
 Zagier \cite{zag94} in 1994, based on the works of Rankin \cite{ran56} and Cohen \cite{coh77}, for any non-negative integer $k$ introduced the $k$-th Rankin-Cohen bracket $[\cdot,\cdot]_k$ defined as follows:
\begin{equation}\label{eq rcb}
  [f,g]_k:=\sum_{i+j=k}(-1)^j\binom{k+r-1}{i}\binom{k+s-1}{j}f^{(j)}g^{(i)} \, , \ \ f\in \mathscr{M}_r  \ and \ g\in \mathscr{M}_s,
\end{equation}
where $f^{(j)}$ and $g^{(j)}$ refer to the $j$-th derivative of $f$ and $g$ with respect to the derivation given in \eqref{eq qmf der}. It was proven by Cohen that $[f,g]_k\in \mathscr{M}_{r+s+2k}$. Note that the $0$-th bracket is considered as usual multiplication, i.e. $[f,g]_0=fg$. We list some  algebraic properties of the Rankin-Cohen brackets given in \cite{zag94} below, in which we assume $f\in \mathscr{M}_r  ,   \ g\in \mathscr{M}_s$ and $h\in \mathscr{M}_l$:
\begin{align}
  &[f,g]_k =(-1)^k[g,f]_k\,, \ \ \forall \, k\geq 0\,,  \label{eq ap1}\\
  &[[f,g]_0,h]_0 =[f,[g,h]_0]_0\,,   \label{eq ap2}\\
  &[f,1]_0 =[1,f]_0=f\,, \ \  [f,1]_k =[1,f]_k=0\,, \ \forall \,k>0\,, \label{eq ap3}\\
  &[[f,g]_1,h]_1+[[g,h]_1,f]_1+[[h,f]_1,g]_1=0\,, \label{eq ap4}\\
  &[[f,g]_0,h]_1+[[g,h]_0,f]_1+[[h,f]_0,g]_1=0\,, \label{eq ap5}\\
  &l[[f,g]_1,h]_0+s[[g,h]_1,f]_0+r[[h,f]_1,g]_0=0\,, \label{eq ap6}\\
  &[[f,g]_0,h]_1=[[g,h]_1,f]_0-[[h,f]_1,g]_1\,, \label{eq ap7}\\
  &(r+s+l)[[f,g]_1,h]_0=r[[g,h]_0,f]_1-s[[h,f]_0,g]_1\,, \label{eq ap8}\\
  &(r+1)(s+1)[[f,g]_0,h]_2=-l(l+1)[[f,g]_2,h]_0  \label{eq ap9}\\
  &\qquad\ +(r+1)(r+s+1)[[g,h]_2,f]_0+(s+1)(r+s+1)[[h,f]_2,g]_0  \nonumber\\
  &(r+s+l+1)(r+s+l+2)[[f,g]_2,h]_0=(r+1)(s+1)[[f,g]_0,h]_2  \label{eq ap10}\\
  &\qquad\ -(r+1)(r+s+1)[[g,h]_0,f]_2-(s+1)(r+s+1)[[h,f]_0,g]_2  \nonumber\\
  &[[f,g]_1,h]_1=[[g,h]_0,f]_2-[[h,f]_0,g]_2+[[g,h]_2,f]_0-[[h,f]_2,g]_0\,. \label{eq ap11}
\end{align}

Zagier defined a Rankin-Cohen algebra over a field $\k$ (of characteristic zero) as a
graded $\k$-vector space $M =\bigoplus_{r\geq 0}M_r$, with $M_0=\k.1$ and $\dim_\k M_r$ finite for all $r$,
together with bilinear operations $[ \ , \ ]_k:M_r\otimes M_s\to M_{r+s+2k},\ r, s, k\geq 0$, which satisfy
\eqref{eq ap1}-\eqref{eq ap11} and all the other algebraic identities satisfied by the Rankin-Cohen brackets given in \eqref{eq rcb}. A basic example of Rankin-Cohen algebras can be constructed as follows, and for future uses we state it as a remark.
\begin{rem} \label{rem srca}
Let $M$ be a commutative and associative graded algebra with unit over the field $\k$ together with a derivation $D$ of degree 2, i.e. $D: M_r \to M_{r+2}$ for all integers $r\geq 0$.
Given $f\in M_r$ and $g\in M_s$, for any non-negative integer $k$ define the Rankin-Cohen bracket $[f,g]_{D,k}$ as follows:
\begin{equation}\label{eq srcb}
  [f,g]_{D,k}=\sum_{i+j=k}(-1)^j\binom{k+r-1}{i}\binom{k+s-1}{j}f^{(j)}g^{(i)}\in M_{r+s+2k},
\end{equation}
where $f^{(j)}=D^{j} f$ and $g^{(j)}=D^{j} g$ are the $j$-th derivative of $f$ and $g$ with respect to the derivation $D$. Then $(M,[\cdot,\cdot]_{D,\ast})$ is a Rankin-Cohen algebra which is called the \emph{standard Rankin-Cohen algebra}.
\end{rem}

For example $(\qmfs,[\cdot,\cdot]_{\rcdomf,\ast})=(\qmfs,[\cdot,\cdot]_{\ast})$, where $\rcdomf$ is given in \eqref{eq qmf der}, is a standard Rankin-Cohen algebras. Hence, $(\mfs,[\cdot,\cdot]_{\rcdomf,\ast})$ is a sub Rankin-Cohen algebra of $(\qmfs,[\cdot,\cdot]_{\rcdomf,\ast})$, but it is not a standard Rankin-Cohen algebras, since $\mfs$ is not closed under $\rcdomf$. We can relate $(\mfs,[\cdot,\cdot]_{\ast})$ with another bilinear form which is defined using the Ramanujan-Serre derivation $\rsdo$. This fact, in a more general version, is given in the following proposition, and since a part of its proof will be needed, we summarize the proof and for more details the reader is referred to the given Ref.

\begin{prop}\label{prop crca} {\rm (\cite[Proposition~1]{zag94})}
Let $M$ be a commutative and associative graded $\k$-algebra with $M_0=\k\cdot 1$ together with a derivation $\rsdo$ of degree 2 on $M$, and let $\Lambda\in M_4$. For any $k\geq 0$ define brackets $[\cdot,\cdot]_{\partial,\Lambda,k}$ by
\begin{equation}\label{eq crcb}
[f,g]_{\partial,\Lambda,k}=\sum_{i+j=k}(-1)^j\binom{k+r-1}{i}\binom{k+s-1}{j}f_{(j)}g_{(i)}  \in M_{r+s+2k} \ ,
\end{equation}
where $f\in M_r, \ g\in M_s,$ and $f_{(j)}\in M_{r+2j}, \ g_{(i)}\in M_{s+2i}$ are defined recursively as follows
\begin{equation}\label{eq f_(i)}
  f_{(j+1)}=\partial f_{(j)}+j(j+r-1)\Lambda f_{(j-1)}, \ g_{(i+1)}=\partial g_{(i)}+i(i+s-1)\Lambda g_{(i-1)} ,
\end{equation}
with initial conditions $f_{(0)}=f, \ g_{(0)}=g$. Then $(M,[\cdot , \cdot]_{\partial,\Lambda,\ast})$ is a Rankin-Cohen algebra.
\end{prop}
{\bf Sketch of proof.}
The only way is to embed $(M,[\cdot , \cdot]_{\partial,\Lambda,\ast})$ into a standard Rankin-Cohen algebra $(R,[\cdot,\cdot]_{D,\ast})$ for some larger $R$ with derivation $D$. Indeed, it is  taken $R=M[\lambda]:=M\otimes_{\k}\k[\lambda]$, where $\lambda\notin M_2$ has degree 2, and the derivation $D$ is defined on the generators of $R$ as follows:
\begin{equation}\label{eq der zag}
  D(f)=\partial(f)+k\lambda f \in R_{k+2}, \ \textrm{for any} \ f\in M_k, \ \textrm{and} \ \ D(\lambda)=\Lambda+\lambda^2\in R_4,
\end{equation}
which can be extended uniquely as a derivation on $R$. Then, for any $k\geq 0$ and any $f\in M_r,\ g\in M_s$, for all $r,s\in \Z_{\geq 0}$, we have:
\begin{equation}\label{eq []=[]}
  [f,g]_{D,k}=[f,g]_{\partial,\Lambda,k}\ \textrm{(see the proof of \cite[Proposition~1]{zag94})}.
\end{equation}
This completes the proof, since $M$ is obviously closed under the brackets $[\cdot,\cdot]_{\partial,\Lambda,k}$. \hfill\(\square\)\\

A Rankin-Cohen algebra $(M,[\cdot , \cdot]_{\ast})$ is called \emph{canonical} if its brackets are given as in
Proposition \ref{prop crca} for some derivation  $\rsdo$ of degree $2$ on M and some element $\Lambda\in M_4$, i.e.,
$[\cdot , \cdot]_k=[\cdot , \cdot]_{\partial,\Lambda,k}$. For example, $(\mfs,[\cdot , \cdot]_{\ast})$ is a canonical
Rankin-Cohen algebra with the Ramanujan-Serre derivation $\partial$ and $\Lambda=\frac{1}{12^2}E_4$.

\section{GMCD for the Dwork family}\label{section MSGMCD}
In Subsections \ref{subsection EMS} and \ref{subsection amsy} we first recall some relevant facts and terminologies from \cite{movnik, younes2}, and for more details one is referred to the same references. Then,
we will observe some new important results in Subsection \ref{subsection R as qhvf} which will be used in the subsequent section. In this manuscript for any positive integer $n$ we fix the notation $\di2:=\frac{n+1}{2}$ if
$n$ is odd, and $\di2:=\frac{n}{2}$ if $n$ is even.

\subsection{Moduli spaces and modular vector field $\Ra$} \label{subsection EMS}
This subsection is based on \cite{movnik}.
 Let $W_z$, for $z\in \P^1\setminus \{0,1,\infty\}$, be an $n$-dimensional hypersurface in $\P^{n+1}$ given by the so-called Dwork family:
\[
f_z(x_0,x_1,\ldots,x_{n+1}):=zx_0^{n+2}+x_1^{n+2}+x_2^{n+2}+\cdots+x_{n+1}^{n+2}-(n+2)x_0
x_1x_2\cdots x_{n+1}=0.
\]
$W_z$ represents a family of \cy $n$-folds. The group $G:=\{(\zeta_0,\zeta_1,\ldots,\zeta_{n+1})\mid \,\, \zeta_i^{n+2}=1,
\ \zeta_0\zeta_1\ldots \zeta_{n+1}=1 \}$,
acts canonically on $W_z$ as
$$
(\zeta_0,\zeta_1,\ldots,\zeta_{n+1}).(x_0,x_1,\ldots,x_{n+1})=(\zeta_0x_0,\zeta_1x_1,\ldots,\zeta_{n+1}x_{n+1}).
$$
We obtain the variety $X=X_z,\ z\in \Pn 1\setminus\{0,1,\infty\}$, by desingularization of the quotient
space $W_z/G $ (for more details see \cite{grmorpl}). From now on, we call $X=X_z$ the \emph{mirror variety}\footnote{The reason for this name is that due to argument given in \cite{grpl90}, the family $X_z$ belongs to the mirror parameter space.} which is also a \cy $n$-fold. It is known that $\dim( H^n_\dR(X))=n+1$ and all Hodge numbers $h^{ij},\ i+j=n,$ of
$X$ are one.

We denote by $\Ts$ the
moduli of the pairs $(X,\alpha_1)$, where $X$ is an $n$-dimensional
mirror variety and $\alpha_1$ is a holomorphic $n$-form on $X$. We
know that the family of mirror varieties $X_z$ is a one parameter
family and the $n$-form $\alpha_1\,$ is unique, up to multiplication
by a constant, therefore $\dim( \Ts)=2$. Analogous to the construction of $X_z$, let $\X_{t_1,t_{n+2}}$,
$(t_1,t_{n+2})\in \C^2 \setminus \{(t_1^{n+2}-t_{n+2})t_{n+2}=0\}$,
be the mirror variety obtained by the quotient and desingularization
of the \cy $n$-folds given by
\begin{equation}\label{eq DF tt}
f_{t_1,t_{n+2}}(x_0,x_1,\ldots,x_{n+1}):=t_{n+2}x_0^{n+2}+x_1^{n+2}+x_2^{n+2}+\cdots+x_{n+1}^{n+2}-(n+2)t_1x_0
x_1x_2\cdots x_{n+1}=0.
\end{equation}
We fix two $n$-forms $\eta$ and $\omega_1$ in the
families $X_z$ and $\X_{t_1,t_{n+2}}$, respectively, such that in the
affine space $\{x_0=1\}$ are given as follows:
\begin{equation}
\eta:=\frac{dx_1\wedge dx_2\wedge \ldots \wedge dx_{n+1}}{df_z} \ ,
\ \ \  \omega_1:=\frac{dx_1\wedge dx_2\wedge \ldots \wedge
dx_{n+1}}{df_{t_1,t_{n+2}}}\ .
\end{equation}
Any element of $\Ts$ is in the form $(X_z,a\eta)$ where $a$ is a
non-zero constant. The pair $(X_z,a\eta)$ can be identified by
$(\X_{t_1,t_{n+2}},\omega_1)$ as follows:
\begin{align}
&(X_z,a\eta)\mapsto
   (\X_{t_1,t_{n+2}},\omega_1)\, , \qquad
   (t_1,t_{n+2})=(a^{-1},za^{-(n+2)}) \, ,\\
   &(\X_{t_1,t_{n+2}},\omega_1)\mapsto
   (X_z,t_1^{-1}\eta) \, ,\qquad z=\frac{t_{n+2}}{t_1^{n+2}}\, .
\end{align}
Hence, $(t_1,t_{n+2})$ construct a chart for $\Ts$; in the other words
\begin{equation}\label{eq MS S}
  \Ts=\spec ( \C[t_1,t_{n+2},\frac{1}{(t_1^{n+2}-t_{n+2})t_{n+2}}])\,,
\end{equation}
and the morphism $\X\to\Ts$ is the universal family of
$(X,\alpha_1)$.
Let
$
 \nabla:H_{\dR}^{n}(\X/\Ts)\to \Omega_\Ts^1\otimes_{\O_\Ts}H_{\dR}^{n}(\X/\Ts)
$ be the Gauss-Manin connection of the  two parameter family of
varieties $\X/\Ts$.  We define the $n$-forms $\omega_i,\,\
i=1,2,\ldots, n+1$, as follows
\begin{equation}
\label{29oct11} \omega_i:= (\nabla_{\frac{\partial}{\partial
t_1}})^{i-1}(\omega_1),
\end{equation}
in which $\frac{\partial}{\partial t_1}$ is considered as a vector
field on the moduli space $\Ts$.
Then $\omega:=\{\omega_1,\omega_2,\ldots,\omega_{n+1}\} $ forms a
basis of $H^{n}_\dR(X)$ which is compatible with its Hodge
filtration, i.e.,
\begin{equation}
\label{eq. Gr. tr.}\omega_i\in F^{n+1-i}\setminus F^{n+2-i},
i=1,2,\ldots, n+1,
\end{equation}
where $F^i$ is the $i$-th piece of the Hodge filtration of
$H^n_\dR(X)$.
We can write the Gauss-Manin connection of $\X/\Ts$ in the basis
$\omega$ as follows
\begin{equation} \label{eq Atilde}
\nabla\omega=\gma \omega\,, {\ \rm with \ } \omega={{\left( {\begin{array}{*{20}{c}}
  {{\omega _1}}&{{\omega _2}}& \ldots &{{\omega _{n + 1}}}
\end{array}} \right)}^{tr}}.
\end{equation}
If we denote by $\gma[i,j]$ the $(i,j)$-th entry of the Gauss-Manin
connection matrix $\gma$, then we obtain:
\begin{align}
&\gma[i,i] = -\frac{i}{(n+2)t_{n+2}}dt_{n+2}\, , \,\ 1 \leq i \leq n\, \label{16/1/2016-1}, \\
&\gma[i,i+1]= dt_1-\frac{t_1}{(n+2)t_{n+2}}dt_{n+2}\, , \,\ 1 \leq i \leq n\, , \label{16/1/2016-2}  \\
&\gma[n+1,j]=\frac{-S_2(n+2,j)t_1^j}{t_1^{n+2}-t_{n+2}}dt_1+\frac{S_2(n+2,j)t_1^{j+1}}{(n+2)t_{n+2}(t_1^{n+2}-t_{n+2})}dt_{n+2}\, , \,\ 1 \leq j \leq n\, , \label{eq stir1} \\
&\gma[n+1,n+1]=\frac{-S_2(n+2,n+1)t_1^{n+1}}{t_1^{n+2}-t_{n+2}}dt_1+
\frac{\frac{n(n+1)}{2}t_1^{n+2}+(n+1)t_{n+2}}{(n+2)t_{n+2}(t_1^{n+2}-t_{n+2})}dt_{n+2}\,, \label{eq stir2}
\end{align}
where $S_2(r,s)$ is the Stirling number of the second kind defined
by \begin{equation} \label{16jan2016} S_2\left( {r,s} \right){\rm{
}} := {\rm{ }}\frac{1}{{s!}}{\rm{ }}\sum\limits_{i = 0}^s {{{( -
1)}^i} \left( {\begin{array}{*{20}{c}}
s\\
i
\end{array}} \right)} {\left( {s - i} \right)^r}\, ,
\end{equation}
and the rest of the entries of $\gma$ are zero.  For any
$\xi_1,\xi_2\in H^n_\dR(X)$, in the context of the de Rham cohomology,
the \emph{intersection form} of $\xi_1$ and $\xi_2$, denoted by
$\langle \xi_1,\xi_2 \rangle$, is given as
$$
\langle \xi_1,\xi_2 \rangle:=\frac{1}{(2\pi i)^n}\int_{X}\xi_1\wedge
\xi_2\,,
$$
which is a non-degenerate $(-1)^n$-symmetric form. We obtain
\begin{align}
  & \langle \omega_i,\omega_j \rangle =0, {\rm \ if} \  i+j\leq n+1\, , \\
  &\langle \omega_1,\omega_{n+1} \rangle
  =(-(n+2))^n\frac{c_n}{t_1^{n+2}-t_{n+2}}, {\rm \ where} \ c_n \  {\rm is \
  a \  constant}\, ,\\
  &\langle \omega_j,\omega_{n+2-j} \rangle =(-1)^{j-1}\langle \omega_1,\omega_{n+1} \rangle, {\rm \ for}  \ j=1,2,\ldots,n+1\, .
\end{align}
On account of these relations, we can determine all the rest of
$\langle\omega_i,\omega_j\rangle$'s in a unique way. If we set
$\Omega=\Omega_n:=\left( \langle
\omega_i,\omega_j\rangle\right)_{1\leq i,j\leq n+1}$ to be the
intersection form matrix in the basis $\omega$, then we have
\begin{equation}
\label{14/12/2015} d\Omega=\gma \Omega+ \Omega \gma^{\tr}.
\end{equation}
For any positive integer $n$ by \emph{moduli space $\T=\T_n$ of enhanced mirror varieties}  we mean the moduli of the pairs $(X,[\alpha_1,\cdots
,\alpha_n,\alpha_{n+1}])$, where $X$ is an $n$-dimensional mirror variety and $\{\alpha_1,\alpha_2,\ldots,\alpha_{n+1}\}$ constructs a basis of $H^n_\dR(X)$
satisfying the properties
\begin{equation}\label{eq bcwhf}
 \alpha_i\in F^{n+1-i}\setminus F^{n+2-i},\ \ i=1,\cdots,n,n+1,
\end{equation}
and
\begin{equation}\label{11jan2016}
[\langle \alpha_i,\alpha_j\rangle]_{1\leq i,j \leq n+1}=\imc_n.
\end{equation}
Here $\imc=\imc_n$ is the following constant $(n+1)\times(n+1)$
matrix:
\begin{equation}\label{eq phi odd}
\imc_n:=\left( {\begin{array}{*{20}c}
   {0_{\di2} } & {J_{\di2}}   \\
   { - J_{\di2}}  & {0_{\di2}}   \\
\end{array}} \right) \, {\rm if \, }n {\rm \, is \, odd,\, and} \,
\imc_n:=J_{n+1} \, {\rm if \, }n {\rm \, is \, even,}
\end{equation}
where by $0_{k}, k\in \mathbb{N},$ we mean a $k\times k$ block of
zeros, $J_1=1$ and
\begin{equation}\small
J_{k }  := \left( {\begin{array}{*{20}c}
   0 & 0 &  \ldots  & 0 & 1  \\
   0 & 0 &  \ldots  & 1 & 0  \\
    \vdots  &  \vdots  &  {\mathinner{\mkern2mu\raise1pt\hbox{.}\mkern2mu
 \raise4pt\hbox{.}\mkern2mu\raise7pt\hbox{.}\mkern1mu}}  &  \vdots  &  \vdots   \\
   0 & 1 &  \ldots  & 0 & 0  \\
   1 & 0 &  \ldots  & 0 & 0  \\
\end{array}} \right), \,\,\textrm{for}\, \, k>1.
\end{equation}
In \cite{movnik} the universal family $\pi:\X\to\T$ together with the global
sections $\alpha_i,\ \ i=1,\cdots, n+1,$ of the relative algebraic de
Rham cohomology $H^n_\dR(\X/\T)$ was constructed, and in its main theorem we observed that:
\begin{theo}{\rm (\cite[Theorem 1.1]{movnik})}
 \label{main3}
There exist a unique vector field $\Ra=\Ra_n \in \mathfrak{X}(\T)$,
and unique regular functions $\Yuk_i\in \O_\T, \ 1\leq i\leq n-2,$
 such that:
\begin{equation}
\label{jimbryan} \nabla_{\Ra}
\begin{pmatrix}
\alpha_1\\
\alpha_2 \\
\alpha_3 \\
\vdots \\
\alpha_n \\
\alpha_{n+1} \\
\end{pmatrix}
= \underbrace {\begin{pmatrix}
0& 1 & 0&0&\cdots &0&0\\
0&0& \Yuk_1&0&\cdots   &0&0\\
0&0&0& \Yuk_2&\cdots   &0&0\\
\vdots&\vdots&\vdots&\vdots&\ddots   &\vdots&\vdots\\
0&0&0&0&\cdots   &\Yuk_{n-2}&0\\
0&0&0&0&\cdots   &0&-1\\
0&0&0&0&\cdots   &0&0\\
\end{pmatrix}}_\Yuk
\begin{pmatrix}
\alpha_1\\
\alpha_2 \\
\alpha_3 \\
\vdots \\
\alpha_n \\
\alpha_{n+1} \\
\end{pmatrix},
\end{equation}
 and $\Yuk\imc+\imc \Yuk^\tr=0$.
\end{theo}
Here $\O_\T$ refers to the $\C$-algebra of regular functions on
$\T$, and $\nabla_{\Ra}$ stands for the algebraic Gauss-Manin connection
$$
\nabla:H_{\dR}^{n}(\X/\T)\to
\Omega_\T^1\otimes_{\O_\T}H_{\dR}^{n}(\X/\T),
$$
composed with the vector field $\Ra\in \mathfrak{X}(\T)$, in which
$\Omega_\T^1$ refers to the $\O_{\T}$-module of differential 1-forms on $\T$.
We call $\Ra$ as \emph{modular vector field} attached to Dwork family. Moreover, we found that:
\begin{equation}
\label{29dec2015} \dt=\dt_n:=\dim ( \T)=\left \{
\begin{array}{l}
\frac{(n+1)(n+3)}{4}+1,\,\, \quad  \textrm{\rm if \textit{n} is odd}
\\\\
\frac{n(n+2)}{4}+1,\,\,\,\,\quad\quad \textrm{\rm if \textit{n} is
even}
\end{array} \right. .
\end{equation}

The above theorem is the key tool of GMCD. In the GMCD viewpoint, the vector field $\Ramvf$ given in \eqref{eq ramanujan}, up to multiplying the coordinates by constants $(t_1,t_2,t_3)=(12t_1,12t_2,\frac{12^3}{8}t_3)$, is the unique vector field that satisfies
\begin{equation}\label{eq gmcd of Ra}
\nabla_\Ramvf\alpha=\left(
                    \begin{array}{cc}
                      0 & 1 \\
                      0 & 0 \\
                    \end{array}
                  \right) \alpha \ ,
\end{equation}
where $\alpha=(\ \alpha_1\ \ \alpha_2\ )^\tr$ and $\nabla$ is the Gauss-Manin connection of the universal family
of elliptic curves
\begin{equation}\label{eq ufec}
y^2=4(x-t_1)^3-t_2(x-t_1)-t_3 \ , \ \ \alpha_1=[\frac{dx}{y}], \ \alpha_2=[\frac{xdx}{y}], \ \text{with} \ \ 27t_3^2-t_2^3\neq 0\ .
\end{equation}
We can generalize the notion of the Ramanujan-Serre derivation \eqref{eq rsd2} and the Rankin-Cohen bracket \eqref{eq rcb}
for the modular vector fields $\Ra=\Ra_n$ using an analogous procedure explained for the Ramanujan vector field $\Ramvf$, which will be treated in Section \ref{section RCACY}.

Next we are going to present a chart for the moduli space $\T$. In order to do
this, let $S=\left(
                           \begin{array}{c}
                             s_{ij} \\
                           \end{array}
                         \right)_{1\leq i,j\leq n+1}$ be a lower triangular matrix, whose  entries are indeterminates $s_{ij},\ \ i\geq j$ and $s_{11}=1$.
We define
\[\underbrace {{{\left( {\begin{array}{*{20}{c}}
  {{\alpha _1}}&{{\alpha _2}}& \ldots &{{\alpha _{n + 1}}}
\end{array}} \right)}^{tr}}}_\alpha  = S\underbrace {\,\,{{\left( {\begin{array}{*{20}{c}}
  {{\omega _1}}&{{\omega _2}}& \ldots &{{\omega _{n + 1}}}
\end{array}} \right)}^{tr}}}_\omega \, ,\]
which implies that $\alpha$  forms a basis of $H^n_\dR(X)$
compatible with its Hodge filtration. We would like that
$(X,[\alpha_1,\alpha_2,\ldots,\alpha_{n+1}])$ be a member of $\T$,
hence it has to satisfy $\left(
                           \begin{array}{c}
                             \langle\alpha_i,\alpha_j\rangle \\
                           \end{array}
                         \right)_{1\leq i,j\leq n+1}=\Phi$, from what we get the following equation
\begin{equation}\label{eq sost}
S\Omega S^\tr=\Phi.
\end{equation}
Using this equation we can express $d_0:=\frac{(n+2)(n+1)}{2}-\dt-2$
numbers of parameters $s_{ij}$'s in terms of other $\dt-2$
parameters that we fix them as \emph{independent parameters}. For simplicity we write the first
class of parameters as $\check t_1,\check t_2,\cdots, \check
t_{d_0}$ and the second class as $t_2,
t_3,\ldots,t_{n+1},t_{n+3},\ldots,t_\dt$. We put the independent
parameters $t_i$ inside $S$ according to the following rule which is not canonical: $t_i$'s  are written in  $S$ from  left to right and top to bottom in the entries $(i,j)$ for $i+j<n+2$ if $n$ is even and $i+j\leq n+2$ if $n$ is odd. The position of $\check t_i$'s inside $S$ can be chosen arbitrarily. For instance, for $n=1,2,3,4,5$ we have:
\[\tiny \left( {\begin{array}{*{20}{c}}
  1&0 \\
  {{t_2}}&{{{\check t}_1}}
\end{array}} \right) , \left( {\begin{array}{*{20}{c}}
1&0&0\\
{{t_2}}&{{\check t}_2}&0\\
{{\check t}_4}&{{\check t}_3}&{{\check t}_1}
\end{array}} \right),\left( {\begin{array}{*{20}{c}}
1&0&0&0\\
{{t_2}}&{{t_3}}&0&0\\
{{t_4}}&{{t_6}}&{{\check t}_2}&0\\
{{t_7}}&{{\check t}_4}&{{\check t}_3}&{{\check t}_1}
\end{array}} \right), \left( {\begin{array}{*{20}{c}}
1&0&0&0&0\\
{{t_2}}&{{t_3}}&0&0&0\\
{{t_4}}&{{t_5}}&{{\check t}_3}&0&0\\
{{t_7}}&{{\check t}_7}&{{\check t}_5}&{{\check t}_2}&0\\
{{\check t}_9}&{{\check t}_8}&{{\check t}_6}&{{\check t}_4}&{{\check
t}_1}
\end{array}} \right), \left( {\begin{array}{*{20}{c}}
  1&0&0&0&0&0 \\
  {{t_2}}&{{t_3}}&0&0&0&0 \\
  {{t_4}}&{{t_5}}&{{t_6}}&0&0&0 \\
  {{t_8}}&{{t_9}}&{{t_{10}}}&{{{\check t}_3}}&0&0 \\
  {{t_{11}}}&{{t_{12}}}&{{{\check t}_7}}&{{{\check t}_5}}&{{{\check t}_2}}&0 \\
  {{t_{13}}}&{{{\check t}_9}}&{{{\check t}_8}}&{{{\check t}_6}}&{{{\check t}_4}}&{{{\check t}_1}}
\end{array}} \right)
.\]
Note that we have already used  $t_1,t_{n+2}$ as coordinate
system of $\Ts$. In particular we find:
\begin{equation}
\label{29/12/2015} s_{(n+2-i)(n+2-i)}=
\frac{(-1)^{n+i+1}}{c_n(n+2)^n}\frac{t_1^{n+2}-t_{n+2}}{s_{ii}}, \
1\leq i \leq m \,.
\end{equation}
In this way, $\t:=(t_1,t_2,\ldots,t_d)$ forms a chart for the moduli space
$\T$, and in fact
\begin{align}
 \label{eq ems}
 \T&=\spec(\C[t_1,t_2,\ldots, t_{\dt},\frac{1}{t_{n+2}(t_{n+2}-t_1^{n+2})\check t }])\,,\\
 \O_\T&=\C[t_1,t_2,\ldots, t_{\dt},\frac{1}{t_{n+2}(t_{n+2}-t_1^{n+2})\check t }]\,\label{eq rfs}.
\end{align}
Here, $\check t$ is the product of $\di2-1$ independent parameters which are located in the main
diagonal of $S$. From now on, we alternately use either $s_{ij}$'s, or $t_i$'s
and ${\check t}_j$'s to refer the entries of $S$. If we denote by
$\gm$  the Gauss-Manin connection matrix of the family $\X/\T$
written in the basis $\alpha$, i.e., $\nabla \alpha=\gm \alpha$,
then we calculate $\gm$ as follows:
\begin{equation}\label{eq gm 3/27/2016}
\gm=\left (dS+S\cdot \gma\right )\, S^{-1}\,.
\end{equation}
If for any vector field $\H\in \mathfrak{X}(\T)$ we define
the \emph{Gauss-Manin connection matrix} attached to $\H$ as $(n+1)\times(n+1)$ matrix $\gm_\H$ given by:
\begin{equation}\label{eq A_H}
\nabla_\H\alpha=\gm_\H \alpha,
\end{equation}
then from \eqref{eq gm 3/27/2016} we obtain:
\begin{equation} \label{eq Sdot}
\dot{S}_\H=\gm_\H S-S\, \gma(\H)\,,
\end{equation}
where $\dot{S}_\H=dS(\H)$  and $\dot x:=dx(\H)$ is the derivative of the function $x$ along
the vector field $\H$ in $\T$. Note that equalities corresponding to $(1,1)$-th and $(1,2)$-th entries of \eqref{eq Sdot} give us respectively $\dot{t}_1$ and $\dot{t}_{n+2}$, and any $\dot{t}_i,\,1\leq i\leq \dt,\, i\neq 1,n+2$, corresponds
to only one $\dot{s}_{jk}$, $1\leq j,k\leq n+1$.
In the following remarks we recall some useful results deduced from the
proof of Theorem \ref{main3} in \cite[\S 7]{movnik}.

\begin{rem}
We obtain the functions $\Yuk_i$'s given in \eqref{jimbryan} as
follows: if $n$ is odd, then
\begin{align}
&\Yuk_{i}=-\Yuk_{n-(i+1)}=\frac{s_{22}\,s_{(i+1)(i+1)}}{s_{(i+2)(i+2)}},\
\ i=1,2,\ldots, \frac{n-3}{2}\, , \label{eq yukiodd}\\
&\Yuk_{\frac{n-1}{2}}=(-1)^{\frac{3n+3}{2}}c_n(n+2)^n\frac{s_{22}\,s_{\frac{n+1}{2}\frac{n+1}{2}}^2}{t_1^{n+2}-t_{n+2}}\,
, \label{eq yukmodd}
\end{align}
and if $n$ is even, then
\begin{align}
&\Yuk_{i}=-\Yuk_{n-(i+1)}=\frac{s_{22}\,s_{(i+1)(i+1)}}{s_{(i+2)(i+2)}},\
\ i=1,2,\ldots, \frac{n-2}{2}\, . \label{eq yukieven}
\end{align}
\end{rem}
\begin{rem}\label{rem uniqueness}
Let $\H\in \mathfrak{X}(\T)$. If $\ \nabla_\H\alpha=0$ for any
$(X,[\alpha_1,\alpha_2,\ldots,\alpha_{n+1}])\in \T$, then $\H=0$.
\end{rem}

We finish this subsection with the following example.

\begin{exam}{\rm
In \cite{movnik} for $n=1,2$  we found the modular vector fields $\Ra_1, \Ra_2$, respectively, as follows:
\begin{equation}
 \label{lovely-1}
 \Ra_{1}:  \left \{ \begin{array}{l} \dot
t_1=-t_1t_2-9(t_1^3-t_3)
\\
\dot t_2=81t_1(t_1^3-t_3)-t_2^2
\\
\dot t_3=-3t_2t_3
\end{array} \right. , \qquad
\Ra_{2}:\left \{ \begin{array}{l}
\dot{t}_1=t_3-t_1t_2\\
\dot{t}_2=2t_1^2-\frac{1}{2}t_2^2\\
\dot{t}_3=-2t_2t_3+8t_1^3\\
\dot{t}_4=-4t_2t_4
\end{array} \right.,
\end{equation}
where by $\dot t_j$ in $\Ra_1$ we mean $\dot t_j=3\cdot q\cdot \frac{\partial t_j}{\partial q}$ and in $\Ra_2$ we mean $\dot t_j=-\frac{1}{5}\cdot q\cdot \frac{\partial t_j}{\partial q}$, and furthermore in $\Ra_2$ we have the polynomial equation $t_3^2=4(t_1^4-t_4)$.
For a complex number $\tau$ with $\rm\, Im \tau>0$, if
we set $q=e^{2\pi i \tau}$, then we obtained the following solutions of
$\Ra_{1}$ and $\Ra_{2}$ respectively:
\begin{equation} \label{eq solution R01}
\left \{ \begin{array}{l}
{\t}_1(q)=\frac{1}{3}(2\theta_3(q^2)\theta_3(q^6)\\ \qquad\qquad\ -\theta_3(-q^2)\theta_3(-q^6)),\\
{\t}_2(q)=\frac{1}{8}(E_2(q^2)-9E_2(q^6)),\\
{\t}_3(q)=\frac{\eta^9({q}^3)}{\eta^3({q})},
\end{array} \right. , \
\left \{ \begin{array}{l}
\frac{10}{6}{\t}_1(\frac{q}{10})=\frac{1}{24}(\theta_3^4(q^2)+\theta_2^4(q^2)),\\\\
\frac{10}{4}{\t}_2(\frac{q}{10})=\frac{1}{24}(E_2(q^2)+2E_2(q^4)),\\\\
10^4{\t}_4(\frac{{q}}{10})=\eta^8({q})\eta^8({q}^2),
\end{array} \right.
\end{equation}
in which  $\eta$ and  $\theta_i$'s  are the classical
eta and theta series given as follows:
\begin{align}
&\eta({q})={q}^{\frac{1}{24}}\prod_{k=1}^\infty (1-{q}^{k}),\,\,\,
\theta_2(q)=\sum_{k=-\infty}^\infty
q^{\frac{1}{2}(\frac{k+1}{2})^2}\,\, , \,\,\,
\theta_3(q)=1+2\sum_{k=1}^\infty q^{\frac{1}{2}k^2}.
\end{align}}
\end{exam}
\subsection{AMSY-Lie algebra and $\sl2$ Lie algebra} \label{subsection amsy}
In this subsection we give a summary of the main results of \cite{younes2}. For any positive integer $n$ the  algebraic
group:
\begin{equation} \label{eq LG}
\LG=\LG_n:=\{\gG\in \textrm{GL}(n+1,\C)\, | \ \gG \rm{\ is  \ upper \
triangular \ and \ } \gG^\tr \Phi \gG=\Phi \},
\end{equation}
acts on the moduli space $\T$ from the right, and its Lie algebra:
\begin{equation} \label{eq LA}
\LA=\{\gL\in{\rm Mat}(n+1,\C)\ | \ \gL \ {\rm is \ upper \
triangular \ and} \ \gL^\tr\Phi+\Phi\gL=0  \}\,,
\end{equation}
is a $\dt-1$ dimensional Lie algebra with the canonical
basis consisting of $\gL_{\a \b}$'s, $1\leq \a \leq \di2, \ \a\leq
\b \leq 2\di2 +1-\a$, given as follows: if $n$ is odd, then
\begin{equation} \label{eq g_ab1}
\gL_{\a\b}={\left( {{g_{kl}}} \right)_{(n + 1) \times (n + 1)}} , \
{\rm where} \left\{ {\begin{array}{*{20}{c}}
{{g_{\a\b}} = 1,\,\,{g_{(n + 2 - \b)(n + 2 - \a)}} =  - 1,{\rm when \ } \b\leq m, }\\
{{g_{\a\b}} ={g_{(n + 2 - \b)(n + 2 - \a)}} =  1,{\rm when \ } \b\geq m+1, }\ \\
{{\rm and \ the \ rest \ of \ the \ entries \ of }\ \gL_{\a\b} {\rm  are \ zero.}\ \ \ \, }
\end{array}} \right.
\end{equation}
and if $n$ is even, then:
\begin{equation} \label{eq g_ab2}
\gL_{\a\b}={\left( {{g_{kl}}} \right)_{(n + 1) \times (n + 1)}} , \
{\rm such \ that \ } \left\{ {\begin{array}{*{20}{c}}
{{g_{\a\b}} = 1,\,\,{g_{(n + 2 - \b)(n + 2 - \a)}} =  - 1,}\qquad\qquad\ \ \ \\
{{\rm and \ the \ rest \ of \ the \ entries \ of }\ \gL_{\a\b} \ {\rm  are \ zero. } \, }
\end{array}} \right.
\end{equation}
The following theorem was proved in \cite{younes2}.
\begin{theo} {\rm (\cite[Theorem 1.2]{younes2})} \label{theo 1}
For any $\gL\in\LA$, there exists a unique vector field $\Ra_\gL\in
\mathfrak{X}(\T)$ such that:
\begin{equation}\label{29jul2017}
\gm_{\Ra_\gL}=\gL^\tr,
\end{equation}
i.e., $\nabla_{\Ra_\gL}\alpha =\gL^\tr\alpha$.
\end{theo}

This theorem yields that the Lie algebra generated by  $\Ra_{\gL_{\a \b}}$'s, $1\leq \a \leq
\di2, \ \a\leq \b \leq 2\di2 +1-\a$, in $\mathfrak{X}(\T)$ with the Lie
bracket of the vector fields is isomorphic to $\LA$ with the Lie bracket of
the matrices. Hence, we use $\LA$ alternately either as a Lie subalgebra
of $\mathfrak{X}(\T)$ or as a Lie subalgebra of ${\rm Mat}(n+1,\C)$.

By \emph{AMSY-Lie algebra}\footnote{The AMSY-Lie algebra was discussed for the first time in \cite{alimov} for non-rigid compact
CY $3$-folds, and in \cite{murmar} it is established for mirror elliptic K3 surfaces. Note
that the AMSY-Lie algebra is called Gauss-Manin Lie algebra by authors of \cite{murmar}.} $\amsy$ we mean the $\O_\T$-module generated by $\LA$
and the modular vector field $\Ra$ in $\mathfrak{X}(\T)$. In what follows, $\delta_j^k$ denotes the Kronecker
delta, $\varrho(n)=1$ if $n$ is an odd integer, and $\varrho(n)=0$
if $n$ is an even integer, $\Yuk_j$'s, $1\leq j\leq n-2$, are the functions given in Theorem \ref{main3}, and besides them we let
$\Yuk_0=-\Yuk_{n-1}:=1$. The following theorem determines the Lie bracket of $\amsy$, which was demonstrated in \cite{younes2}.

\begin{theo} {\rm (\cite[Theorem 1.3]{younes2})} \label{theo 2}
The following hold:
 \begin{align}
&{[\Ra,\Ra_{\gL_{11}}]}=\Ra, \label{eq Rag11}\\
&{[\Ra,\Ra_{\gL_{22}}]}=-\Ra, \label{eq Rag22} \\
&{[\Ra,\Ra_{\gL_{\a\a}}]}=0, \ 3\leq \a \leq m\, , \label{eq Ragaa} \\
&[\Ra,{\Ra_{\gL_{\a\b}}}]=\Psi_1^{\a\b}(\Yuk)\,{\Ra_{\gL_{(\a+1)\b}}}+\Psi_2^{\a\b}(\Yuk)\,{\Ra_{\gL_{\a(\b-1)}}},\,
 1\leq \a \leq \di2, \ \a+1\leq \b \leq 2\di2+1-\a \, , \label{eq Ragab}
\end{align}
where
\begin{align}
&\Psi_1^{\a\b}(\Yuk):=(1+\varrho(n)\delta_{\a+\b}^{2\di2}-\delta_{\a+\b}^{2\di2+1})\,\Yuk_{\a-1},\\
&\Psi_2^{\a\b}(\Yuk):= (1-2\varrho(n)\delta_{\b}^{\di2+1})\,\Yuk_{n+1-\b}\,.
\end{align}
\end{theo}

If $n=1,2$, then we see that $\amsy$ is isomorphic to
$\mathfrak{sl}_2(\C)$. In general, for $n\geq 3$ we
have a copy of $\mathfrak{sl}_2(\C)$ as a Lie subalgebra of $\amsy$ which contains the modular vector field $\Ra$ and we state it in the following
theorem from Ref. \cite{younes2}.
\begin{theo} {\rm (\cite[Theorem 1.3]{younes2})}  \label{theo 3}
Let us define the vector fields $\rvf$ and $\cvf$ as follows:
\begin{enumerate}
  \item if $n=1$, then $\rvf:=-\Ra_{\gL_{11}}$ and $\cvf:=\Ra_{\gL_{12}}$,
  \item if $n=2$, then $\rvf:=-2\Ra_{\gL_{11}}$ and $\cvf:=2\Ra_{\gL_{12}}$,
  \item if $n\geq 3$, then $\rvf:=\Ra_{\gL_{22}}-\Ra_{\gL_{11}}$ and $\cvf:=\Ra_{\gL_{12}}$.
\end{enumerate}
Then the Lie algebra generated by the vector fields
$\Ra, \rvf, \cvf$ in $\amsy\subset\mathfrak{X}(\T)$ is isomorphic to
$\mathfrak{sl}_2(\C)$; indeed we get:
\[
[\Ra,\cvf]=\rvf \ , \ \ [\rvf,\Ra]=2\Ra \ , \ \ [\rvf,\cvf]=-2\cvf
\, .
\]
\end{theo}

According to Theorem \ref{theo 3}, if $n=1,2$, then $\amsy$ is isomorphic to
$\mathfrak{sl}_2(\C)$ (see Example \ref{exam 1}), and for $n\geq 3$ the Lie subalgebra of $\amsy$ generated by
$\Ra$, $\rvf:=\Ra_{\gL_{22}}-\Ra_{\gL_{11}}$ and $\cvf:=\Ra_{\gL_{12}}$ is isomorphic to $\mathfrak{sl}_2(\C)$. Using the equalities corresponding to $(1,1)$-th and $(1,2)$-th entries of \eqref{eq Sdot} for the vector fields $\Ra_{\gL_{\a\b}}$'s we obtain
the diagonal matrix $\gma(\Ra_{\gL_{11}})=\diag(1,2,\ldots,n+1)$ and the null matrices $\gma(\Ra_{\gL_{\a\b}})=0$, for $1\leq \a\leq m, \, \a\leq
\b\leq 2m+1-\a,\, \b\neq 1$ (see \cite[\S~4.4]{younes2}). Due to these facts and again \eqref{eq Sdot}, we can find $\dot{S}_{\Ra_{\gL_{\a\b}}}$'s, and consequently we obtain $\Ra_{\gL_{\a\b}}$'s.
In particular, knowing that $\dot{S}_{\rvf}=\dot{S}_{\Ra_{\gL_{22}}}-\dot{S}_{\Ra_{\gL_{11}}}$, we get $dt_1(\rvf)=t_1,\, dt_{n+2}(\rvf)=(n+2)t_{n+2}$, and hence
\begin{equation} \label{eq SdotH}
\dot{S}_{\rvf}={\tiny \left( {\begin{array}{*{20}{c}}
0&0&0&0& \ldots &0&0&0\\
{2{s_{21}}}&{3{s_{22}}}&0&0& \ldots &0&0&0\\
{{s_{31}}}&{2{s_{32}}}&{3{s_{33}}}&0& \ldots &0&0&0\\
{{s_{41}}}&{2{s_{42}}}&{3{s_{43}}}&{4{s_{44}}}& \ldots &0&0&0\\
 \vdots & \vdots & \vdots & \vdots & \ddots & \vdots & \vdots & \vdots \\
{{s_{(n - 1)1}}}&{2{s_{(n - 1)2}}}&{3{s_{(n - 1)3}}}&{4{s_{(n - 1)4}}}& \ldots &{(n - 1){s_{(n - 1)(n - 1)}}}&0&0\\
0&{{s_{n2}}}&{2{s_{n3}}}&{3{s_{n4}}}& \ldots &{(n - 2){s_{n(n - 1)}}}&{(n - 1){s_{nn}}}&0\\
{2{s_{(n + 1)1}}}&{3{s_{(n + 1)2}}}&{4{s_{(n + 1)3}}}&{5{s_{(n + 1)4}}}& \ldots &{n{s_{(n + 1)(n - 1)}}}&{(n + 1){s_{(n + 1)n}}}&{(n + 2){s_{(n + 1)(n + 1)}}}
\end{array}} \right)\,.}
\end{equation}
Thus, for an even integer $n\geq 5$ we get:
{\small
\begin{align}
\rvf&=t_1\frac{\partial}{\partial t_1}+2t_2\frac{\partial}{\partial t_2}+3t_3\frac{\partial}{\partial t_3}+\sum_{ \ i = 4\hfill\atop
i \ne n + 2\hfill}^{\dt-1}w_it_i\frac{\partial}{\partial t_i}+(n+2)t_{n+2}\frac{\partial}{\partial t_{n+2}}+\frac{n+2}{2}t_{\dt+1}\frac{\partial}{\partial t_{\dt+1}}\,, \label{eq rvf even} \\
\cvf&=\frac{\partial}{\partial t_2}\,,\label{eq cvf even}
\end{align}}
with $t_{\dt+1}^2=s_{\frac{n+2}{2}\frac{n+2}{2}}^2=\frac{(-1)^{\frac{n}{2}}}{c_n(n+2)^n}(t_1^{n+2}-t_{n+2})$ (see \eqref{29/12/2015}), and for an odd integer $n\geq 5$ we obtain:
{\small
\begin{align}
\rvf&=t_1\frac{\partial}{\partial t_1}+2t_2\frac{\partial}{\partial t_2}+3t_3\frac{\partial}{\partial t_3}+\sum_{ \ i = 4\hfill\atop
i \ne n + 2\hfill}^{\dt-3}w_it_i\frac{\partial}{\partial t_i}+(n+2)t_{n+2}\frac{\partial}{\partial t_{n+2}}+t_{\dt-1}\frac{\partial}{\partial t_{\dt-1}}+2t_{\dt}\frac{\partial}{\partial t_{\dt}}\,, \label{eq rvf odd} \\
\cvf&=\frac{\partial}{\partial t_2}-t_{\dt-2}\frac{\partial}{\partial t_\dt}\,.\label{eq cvf odd}
\end{align}}
In both equations \eqref{eq rvf even} and \eqref{eq rvf odd} we have $w_i=k  \ \textrm{if} \ t_i=s_{jk}$ for some $1\leq j,k \leq n+1$, i.e., $w_i$ is the number of the column of the entry $t_i$. Note that $\rvf$ and $\cvf$ have been computed explicitly for $n=1,2,3,4$ in Example \ref{exam 1}, which are similar to the $\rvf$ and $\cvf$ founded above for the cases $n\geq 5$. Hence, in general we can write $\rvf$ as:
\begin{equation}\label{eq rvf gf}
  \rvf=\sum_{i=1}^{\dt}w_it_i\frac{\partial}{\partial t_i}\,,
\end{equation}
where $w_i$'s are non-negative integers.
\begin{rem}\label{rem weights}
\begin{enumerate}
  \item If $n=1$, then $w_1=1,\, w_2=2,\, w_3=3$.
  \item If $n=2$, then $w_1=2,\, w_2=2,\, w_4=8$.
  \item If $n=3$, then $w_1=1,\, w_2=2,\, w_3=3,\, w_4=0,\, w_5=5,\, w_6=1,\, w_7=2$.
  \item If $n\geq 4$ is an even integer, then $w_1=1,\, w_2=2,\, w_3=3,\,w_{n+2}=n+2,\,w_\dt=0$.
  \item If $n\geq 5$ is an odd integer, then $w_1=1,\, w_2=2,\, w_3=3,\,w_{n+2}=n+2,\,w_{\dt-2}=0,\,w_{\dt-1}=1,\,w_{\dt}=2$.
\end{enumerate}
\end{rem}

\subsection{$\Ra$ as a quasi-homogeneous vector field} \label{subsection R as qhvf}
Let us attach to any $t_i$ in $\O_\T$ the weight $\deg(t_i)=w_i$, in which the non-negative integers $w_i$'s are given in \eqref{eq rvf gf}. Recall that a vector field $\H=\sum_{j=1}^{\dt}\H^j\frac{\partial}{\partial t_j}\in \mathfrak{X}(\T)$, with $\H^j\in \O_\T$, is said to be \emph{quasi-homogeneous of degree} $d$ if
for any $1\leq j\leq \dt$ we have $\deg(\H^j)=w_j+d$. Hence, on account of \eqref{eq rvf even}, \eqref{eq cvf even}, \eqref{eq rvf odd}, \eqref{eq cvf odd} and Remark \ref{rem weights} the vector fields $\rvf$ and $\cvf$ are quasi-homogeneous of degree $0$ and $-2$, respectively. The vector field $\rvf$ is also known as the radial
vector field. Moreover, in the following proposition we show that $\Ra$ is a quasi-homogeneous vector field as well.

\begin{prop}\label{prop Risqh2}
The modular vector field $\Ra$ is a quasi-homogeneous vector field of degree $2$ on $\T$.
\end{prop}
\begin{proof}
Due to Example \ref{exam 1} the affirmation is valid for $n=1,2,3,4$. Hence we suppose that $n\geq 5$. First note that in the proof of Theorem \ref{theo 1} (see \cite[\S~4.1]{younes2}) it is verified that the equations $S\Omega S^\tr=\imc$ and  $\dot{S}_\gL=\gm_\gL S-S\, \gma(\gL)$ are compatible for any $\gL\in \LA$. In particular, it holds for $\gL=\rvf$. This implies that the degree of any entry $s_{jk}$ of $S$, $2\leq j\leq n+1,\, 1\leq k \leq j$,  is equal to the integer multiple of $s_{jk}$ in the matrix $\dot{S}_\rvf$, which is stated in \eqref{eq SdotH}. If we set $\Ra=\sum_{i=1}^{\dt}\dot{t}_{i} \frac{\partial}{\partial t_i}$, then $\dot{t}_i$'s follow from
\begin{equation}\label{eq SdotR}
  \dot{S}_\Ra=\Yuk S-S\, \gma(\Ra)\,.
\end{equation}
More precisely, from the equalities corresponding to $(1,1)$-th and $(1,2)$-th entries of \eqref{eq SdotR} we obtain:
\begin{equation} \label{eq t1dtn+2d}
\dot{t}_1=s_{22}-t_1s_{12} \ \ \& \ \ \dot{t}_{n+2}=-(n+2)s_{21}t_{n+2}\,.
\end{equation}
These equalities and \eqref{16/1/2016-1}-\eqref{eq stir2} imply:
\begin{align*}
&\left(-\frac{k}{(n+2)t_{n+2}}dt_{n+2}\right)(\Ra)=ks_{21},\,\, 1\leq k \leq n\,,\\
&\left(dt_1-\frac{t_1}{(n+2)t_{n+2}}dt_{n+2}\right)(\Ra)=s_{22}\,,\\
&\left(\frac{-S_2(n+2,j)t_1^j}{t_1^{n+2}-t_{n+2}}dt_1+\frac{S_2(n+2,j)t_1^{j+1}}{(n+2)t_{n+2}(t_1^{n+2}-t_{n+2})}dt_{n+2}\right)(\Ra)=
\frac{-S_2(n+2,j)t_1^js_{22}}{t_1^{n+2}-t_{n+2}}\,,
\end{align*}
\begin{align*}\label{2110}
  \left(\frac{-S_2(n+2,n+1)t_1^{n+1}}{t_1^{n+2}-t_{n+2}}dt_1 \right. & \left. +\frac{\frac{n(n+1)}{2}t_1^{n+2}+(n+1)t_{n+2}}{(n+2)t_{n+2}(t_1^{n+2}-t_{n+2})}dt_{n+2}\right)(\Ra)\\
  & =(n+1)s_{21}-\frac{(n+1)(n+2)}{2}\frac{t_1^{n+1}s_{22}}{t_1^{n+2}-t_{n+2}}\,.\qquad\qquad\ \,
\end{align*}
Note that in the above last equality we used the fact that $S_2(n+2,n+1)=\frac{(n+1)(n+2)}{2}$.
Therefore:
\[\gma(\Ra)={\tiny \left( {\begin{array}{*{20}{c}}
{{s_{21}}}&{{s_{22}}}&0&0&0\\
0&{2{s_{21}}}&{{s_{22}}}&\ldots&0\\
 \vdots & \vdots & \ddots & \ddots & \vdots \\
0&0&\ldots&{n{s_{21}}}&{{s_{22}}}\\
\frac{-S_2(n+2,1)t_1s_{22}}{t_1^{n+2}-t_{n+2}}&\frac{-S_2(n+2,2)t_1^2s_{22}}{t_1^{n+2}-t_{n+2}}&\ldots&\frac{-S_2(n+2,n)t_1^ns_{22}}{t_1^{n+2}-t_{n+2}}&
(n+1)s_{21}-\frac{(n+1)(n+2)}{2}\frac{t_1^{n+1}s_{22}}{t_1^{n+2}-t_{n+2}}
\end{array}} \right)}\,,\]
hence, $S\,\gma(\Ra)$ equals
\begin{equation}\label{eq SB(R)}
{\tiny \left(
  \begin{array}{ccccccc}
    s_{21} & s_{22} & 0 & 0 & \ldots & 0 & 0 \\
    s_{21}s_{21} & s_{21}s_{22}+2s_{22}s_{21} & s_{22}s_{22} & 0 & \ldots & 0 & 0  \\
    s_{31}s_{21} & s_{31}s_{22}+2s_{32}s_{21} & s_{32}s_{22}+3s_{33}s_{21} & s_{33}s_{22} & \ldots & 0 & 0  \\
    s_{41}s_{21} & s_{41}s_{22}+2s_{42}s_{21} & s_{42}s_{22}+3s_{43}s_{21} & s_{43}s_{22}+ 4s_{44}s_{21}& \ldots & 0 & 0  \\
    \vdots & \vdots & \vdots & \vdots & \ddots & \vdots & \vdots  \\
    s_{n1}s_{21} & s_{n1}s_{22}+2s_{n2}s_{21} & s_{n2}s_{22}+3s_{n3}s_{21} & s_{n3}s_{22}+4s_{n4}s_{21} & \ldots & s_{n(n-1)}s_{22}+ns_{nn}s_{21} & s_{nn}s_{22}  \\
    S\gma(\Ra)[n+1,1] & S\gma(\Ra)[n+1,2] & S\gma(\Ra)[n+1,3] & S\gma(\Ra)[n+1,4] & \ldots & S\gma(\Ra)[n+1,n] & S\gma(\Ra)[n+1,n+1]  \\
  \end{array}
\right)}\,,
\end{equation}
in which:
\begin{align*}
&S\gma(\Ra)[n+1,1]=s_{(n+1)1}s_{21}-\frac{S_2(n+2,1)t_1s_{22}s_{(n+1)(n+1)}}{t_1^{n+2}-t_{n+2}}\,,\\
&S\gma(\Ra)[n+1,j]=s_{(n+1)(j-1)}s_{22}+js_{(n+1)j}s_{21}-\frac{S_2(n+2,j)t_1^js_{22}s_{(n+1)(n+1)}}{t_1^{n+2}-t_{n+2}}\,, \ \ 2\leq j \leq n \,,\\
&S\gma(\Ra)[n+1,n+1]=s_{(n+1)n}s_{22}+s_{(n+1)(n+1)}\left( (n+1)s_{21}-\frac{(n+1)(n+2)}{2}\frac{t_1^{n+1}s_{22}}{t_1^{n+2}-t_{n+2}} \right) \,.
\end{align*}
Observe that
\begin{equation}\label{eq yuk mat}
\Yuk S={\tiny \left( {\begin{array}{*{20}{c}}
{{s_{21}}}&{{s_{22}}}&0&0& \ldots &0&0\\
{{\Yuk_1}{s_{31}}}&{{\Yuk_1}{s_{32}}}&{{\Yuk_1}{s_{33}}}&0& \ldots &0&0\\
{{\Yuk_2}{s_{41}}}&{{\Yuk_2}{s_{42}}}&{{\Yuk_2}{s_{43}}}&{{\Yuk_2}{s_{44}}}& \ldots &0&0\\
 \vdots & \vdots & \vdots & \vdots & \ddots & \vdots & \vdots \\
{{\Yuk_{n - 2}}{s_{n1}}}&{{\Yuk_{n - 2}}{s_{n2}}}&{{\Yuk_{n - 2}}{s_{n3}}}&{{\Yuk_{n - 2}}{s_{n4}}}& \ldots &{{\Yuk_{n - 2}}{s_{nn}}}&0\\
{ - {s_{(n + 1)1}}}&{ - {s_{(n + 1)2}}}&{ - {s_{(n + 1)3}}}&{ - {s_{(n + 1)4}}}& \ldots &{ - {s_{(n + 1)n}}}&{ - {s_{(n + 1)(n + 1)}}}\\
0&0&0&0& \ldots &0&0
\end{array}} \right)\,,}
\end{equation}
and \eqref{eq yukiodd}-\eqref{eq yukieven} imply that $\deg(\Yuk_1)=\deg(\Yuk_{n-2})=3$ and $\deg(\Yuk_j)=2,\, 2\leq j\leq n-3$. If we denote the $(i,j)$-th entry of $\dot{S}_\Ra$ by $\dot{S}_\Ra[i,j]$, then \eqref{eq SdotR}, \eqref{eq SB(R)} and \eqref{eq yuk mat} yield $\deg(\dot{S}_\Ra[i,j])=\deg(s_{ij})+2,\, 2\leq i \leq n+1,\, 1\leq j \leq i$, which complete the proof.
\end{proof}
\begin{rem} \label{rem comp vf R}
Using the matrix $\dot{S}_\Ra=\Yuk S-S\, \gma(\Ra)$ computed in the proof of the above proposition we can encounter the modular
vector field $\Ra$ explicitly for any $n\geq 5$.
\end{rem}

We can prove Proposition \ref{prop Risqh2} in a simpler way, but we can not compute $Ra$ explicitly in this case. Below we give this proof as well.
\bigskip\\
{\it Another proof of Proposition \ref{prop Risqh2}.} Since $\rvf=\sum_{j=1}^{\dt}w_j \frac{\partial}{\partial t_j}$ and $\deg(t_j)=w_j$, we can easily observe that for a given $f\in \O_\T$ we have $\rvf (f)=kf$, for a $k\in \Z$, if and only if $f$ is a quasi-homogeneous element of degree $k$ of $\O_\T$. According to $[\rvf,\Ra]=2\Ra$, for any quasi-homogenous element $f\in \T$ of degree $k\in \Z$ we have:
\[
[\rvf,\Ra](f)=2\Ra(f) \Rightarrow \rvf(\Ra(f))-\Ra(\rvf(f))=2\Ra(f) \Rightarrow \rvf(\Ra(f))=(k+2)\Ra(f),
\]
which implies $\Ra(f)$ is a quasi-homogeneous element of degree $k+2$, and this is equivalent to say that $\Ra$ is a quasi-homogeneous vector field of degree $2$. \hfill\(\square\)

 \bigskip
The following lemma is useful for the future use.
\begin{lemm}\label{lemm R^2}
If we write
\[
\Ra=\sum_{j=1}^{\dt}\Ra^j(t_1,t_2,\ldots,t_\dt)\frac{\partial}{\partial t_j}, \ \textrm{with} \ \ \Ra^j\in \O_\T\,,
\]
and define
\begin{equation}\label{eq Lambda}
\Lambda(t_1,t_2,\ldots,t_\dt):= \left\{
                                  \begin{array}{ll}
                                    -\frac{1}{2}\Ra^2(t_1,t_2,\ldots,t_\dt)-\frac{1}{4}t_2^2, &if\ n=2; \hbox{} \\\\
                                    -\Ra^2(t_1,t_2,\ldots,t_\dt)-t_2^2, &if\ n \neq 2, \hbox{}
                                  \end{array}
                                \right.
\end{equation}
then $\deg (\Lambda)=4$ and $\frac{\partial \Lambda}{\partial t_2}=0$.
\end{lemm}
\begin{proof}
  For $n=1,2,3,4$ the modular vector field $\Ra$ has been explicitly stated in Example \ref{exam 1} and one can easily check the truth of the statement. For $n\geq 5$ the component $\Ra^2$ of the modular vector field $\Ra$ corresponds to the $(2,1)$-th entry of the matrix $\dot{S}_\Ra=\Yuk S-S\, \gma(\Ra)$ computed in the proof of Proposition \ref{prop Risqh2} that yields:
\[
\Ra^2(t_1,t_2,\ldots,t_\dt)=\Yuk_1t_4-t_2^2\,,\ \ \textrm{(note that $t_2=s_{21}$ and $t_4=s_{31}$)}.
\]
From \eqref{eq yukiodd} and \eqref{eq yukieven} we get $\Yuk_1=\frac{s_{22}^2}{s_{33}}=\frac{t_3^2}{t_6}$, which implies:
\[
\Ra^2(t_1,t_2,\ldots,t_\dt)=\frac{t_3^2t_4}{t_6}-t_2^2\,.
\]
Hence, for $n\geq 5$ we obtain $\Lambda=-\frac{\ t_3^2t_4}{t_6}$ and the proof is complete.
\end{proof}

\section{Vector field $\rcv$} \label{subsection rcv}
Remember from Section \ref{section introduction} that the Ramanujan vector field $\Ramvf=\frac{1}{12}(t_1^2-t_2)\frac{\partial}{\partial t_1}+\frac{1}{3}(t_1t_2-t_3)\frac{\partial}{\partial t_2}+\frac{1}{2}(t_1t_3-t_2^2)\frac{\partial}{\partial t_3}$ along with $H=2t_1\frac{\partial}{\partial t_1}+4t_2\frac{\partial}{\partial t_2}+6t_3\frac{\partial}{\partial t_3}$ and $F=-12\frac{\partial}{\partial t_1}$ generates $\sl2$ and we have $F(t_1)=-12\neq 0, \ F(t_2)=F(t_3)=0$. In the case of modular vector field, we observed in Theorem \ref{theo 3} that the Lie algebra generated by $\Ra, \rvf, \cvf$ is isomorphic to $\sl2$. According to Example \ref{exam 1}, \eqref{eq cvf even} and \eqref{eq cvf odd}, we find the vector field $\cvf$ for any positive integer $n$ as follows:
\begin{align}
  \cvf &= (1+\delta_2^n)\frac{\partial}{\partial t_2}\,, \ \ {\rm if} \ n=1 \ {\rm or} \ n \ {\rm is \ even},  \label{eq cvf n=1,even} \\
  \cvf &= \frac{\partial}{\partial t_2}-t_{4}\frac{\partial}{\partial t_7} \,, \ \ {\rm if} \ n=3\,,  \label{eq cvf n=3}\\
  \cvf &= \frac{\partial}{\partial t_2}-t_{\dt-2}\frac{\partial}{\partial t_\dt}\,, \ \ {\rm if}  \ n\geq 5 \ {\rm is \ odd}.  \label{eq cvf n=odd>3}
\end{align}
For any $n$ we have $\cvf(t_2)=1+\delta_2^n\neq 0$. If $n=1$ or $n$ is even, then $\cvf(t_j)=0$ for all $1\leq j \leq \dt$ and $j\neq 2$. But if $n\geq 3$ is odd, then, besides $\cvf(t_2)\neq 0$, we have $\cvf(t_7)=-t_4\neq 0$, when $n=3$, and $\cvf(t_\dt)=-t_{\dt-2}\neq 0$, when $n\geq 5$. These will cause problems for our purposes in Section \ref{section RCACY}, when $n\geq 3$ is odd. To overcome these problems we have the following two options.
\begin{description}
  \item[Option 1.] In this option we change the chart of $\T$, but $\Ra, \rvf, \cvf$ stay the same. If for any positive odd integer $n\geq 3$ we set:
  \[
  \tilde{t}_\dt:=\left\{
                 \begin{array}{ll}
                   t_7+t_2t_4, & \hbox{{\rm if} \ $n=3$;} \\
                   t_\dt+t_2t_{\dt-2}, & \hbox{{\rm if} \ $n\geq 5$ \ {is \ odd};}
                 \end{array}
               \right.,
  \]
then it is easy to observe that $\cvf(\tilde{t}_\dt)=0$. According to \eqref{eq rvf odd}, we observe that $\rvf(\tilde{t}_\dt)=2$, which means $\deg(\tilde{t}_\dt)=2$. If instead the chart $(t_1,t_2,\ldots,t_\dt)$ for $\T$, we use the chart $(t_1,t_2,\ldots,t_{\dt-1},\tilde{t}_\dt)$, and set $\rcv:=\Ra=\sum_{j=1}^{\dt}\Ra^j\frac{\partial}{\partial t_j}$, then we have:
{\small
\begin{align}
  \rcv & = \sum_{j=1}^{\dt-1}\rcv^j\frac{\partial}{\partial t_j}+\rcv^\dt\frac{\partial}{\partial \tilde{t}_\dt},  \ {\rm where} \ \rcv^j=\Ra^j, \ j=1,2,\ldots,\dt-1 \ {\rm and} \ \rcv^\dt=\Ra(\tilde{t}_\dt),\\
  \rvf & = \sum_{j=1}^{\dt-1}w_j t_j \frac{\partial}{\partial t_j}+2\tilde{t}_\dt \frac{\partial}{\partial \tilde{t}_\dt}, \ \ {\rm where} \ w_j, \ j=1,2,\ldots,\dt-1, \ {\rm \  are \ as \ before}, \\
  \cvf & = \frac{\partial}{\partial t_2}\,.
\end{align}}
For simplicity, from now on, we denote $\tilde{t}_\dt$ also by $t_d$, but we remember that whenever we use $\rcv$ of this option we consider the chart
$(t_1,t_2,\ldots,t_{\dt-1},\tilde{t}_\dt)$, and $\rcv=\Ra$, $\rvf$, $\cvf$ does not change.
  \item[Option 2.] In this option we change the modular vector field $\Ra$, and consequently $\cvf$, but the chart of $\T$ and $\rvf$ remain the same. For any positive integer $n$ we define:
\begin{equation}\label{eq rcv}
  \rcv:=\Ra+t_2\left( [\Ra,(1+\delta_2^n)\frac{\partial}{\partial t_2}]-\rvf\right)\,.
\end{equation}
Note that if $n=1$ or $n$ is even, then $(1+\delta_2^n)\frac{\partial}{\partial t_2}=\cvf$, which implies $\rcv=\Ra$. But, if $n\geq 3$ is odd, then $\rcv$ is different from $\Ra$.
In the next lemmas we show that $\rcv$ is a quasi-homogeneous vector field of degree $2$, and along with $\rvf$ and $(1+\delta_2^n)\frac{\partial}{\partial t_2}$ forms a copy of $\sl2$.
\end{description}

\begin{rem}
\begin{enumerate}
  \item The vector field $\rcv$ defined in {\rm {\bf Option 1}} is different from the one defined in {\rm {\bf Option 2}} whenever $n\geq 3$ is odd. The reason for denoting them using the same notation is that the main results of this paper which will be proved in Section \ref{section RCACY} hold for both of them, and to avoid repeating the theorems and their proofs we use the same notation.
  \item Since in {\rm {\bf Option 1}} we have $\rcv=\Ra$, we can find the $q$-expansion of their solution components (at least for $n=1,2,3,4$) and we know interesting facts about them, some of which were mentioned in Section \ref{section introduction}. But, we do not know solutions of $\rcv$ given in {\rm {\bf Option 2}} when $n\geq 3$ is odd. In particular, the author tried to find the $q$-expansion of a solution of it for $n=3$, but he did not succeed and it seems that the solutions of $\rcv$ in these cases does not have $q$-expansion around $\infty$.
\end{enumerate}
\end{rem}
Next we state the fundamental lemma of this work, which will be used to prove Theorem \ref{theo 4}. First, we recall that if we have two vector fields $V=\sum_{j=1}^{\dt}V^j\frac{\partial}{\partial t_j}$ and $W=\sum_{j=1}^{\dt}W^j\frac{\partial}{\partial t_j}$, then
\begin{equation}\label{eq lie bracket in a chart}
[V,W]=VW-WV=\sum_{j=1}^{\dt}\left( V(W^j)-W(V^j) \right) \frac{\partial}{\partial t_j}\,.
\end{equation}

\begin{lemm}\label{lemm fund}
{\bf (Fundamental lemma)} The vector field $\rcv$ is a quasi-homogeneous vector field of degree $2$ in the AMSY-Lie algebra $\amsy$ that satisfies:
\begin{equation}\label{eq [D,d/dt2]}
  [\rcv,(1+\delta_2^n)\frac{\partial}{\partial t_2}]=\rvf\,.
\end{equation}
\end{lemm}
\begin{proof} If $\rcv$ is the one given in {\rm {\bf Option 1}}, then $\rcv=\Ra$ and $\cvf=(1+\delta_2^n)\frac{\partial}{\partial t_2}$. Hence, due to Proposition \ref{prop Risqh2} and Theorem \ref{theo 3} the lemma holds.

Now we suppose that $\rcv$ is the one defined in {\rm {\bf Option 2}}.
If $n=1,2,3,4$, then $\Ra,\, \cvf,\, \rvf$ are given explicitly in Example \ref{exam 1}, and one can easily find that the affirmations hold.
If $n\geq5$ is even, then $\rcv=\Ra$ and $\cvf=\frac{\partial}{\partial t_2}$, which yield the results. Suppose that $n\geq5$ is odd. Then by applying \eqref{eq Sdot} to $\Ra_{\gL_{1n}}$ and $\Ra_{\gL_{1(n+1)}}$ we obtain   $\Ra_{\gL_{1n}}=\frac{\partial}{\partial t_{\dt-2}}+t_2\frac{\partial}{\partial t_{\dt}}$ and $\Ra_{\gL_{1(n+1)}}=\frac{\partial}{\partial t_\dt}$.
      Therefore, the relation \eqref{eq Ragab} yields:
\begin{equation}\label{eq 29/1/2019}
[\Ra,\frac{\partial}{\partial t_\dt}]=[\Ra,\Ra_{\gL_{1(n+1)}}]=\Ra_{\gL_{1n}}=\frac{\partial}{\partial t_{\dt-2}}+t_2\frac{\partial}{\partial t_{\dt}}\, .
\end{equation}
If we write $\Ra=\sum_{j=1}^{\dt}\Ra^j\frac{\partial}{\partial t_j}$, then Remark \ref{rem comp vf R} yields  $\Ra^{\dt-2}=-t_\dt-t_2t_{\dt-2}$, from which we get:
\begin{align}
 [\Ra,t_{\dt-2}\frac{\partial}{\partial t_\dt}]&=\Ra(t_{\dt-2})\frac{\partial}{\partial t_\dt}+t_{\dt-2}[\Ra,\frac{\partial}{\partial t_\dt}] \label{eq 28/11/2018}\\
 &\mathop = \limits^{\eqref{eq 29/1/2019}}\Ra^{\dt-2}\frac{\partial}{\partial t_\dt}+t_{\dt-2}\frac{\partial}{\partial t_{\dt-2}}+t_2t_{\dt-2}\frac{\partial}{\partial t_\dt} \nonumber\\
&=t_{\dt-2}\frac{\partial}{\partial t_{\dt-2}}-t_{\dt}\frac{\partial}{\partial t_\dt}\,.\nonumber
\end{align}
Due to \eqref{eq cvf odd} we have $\frac{\partial}{\partial t_2}=\cvf+t_{\dt-2}\frac{\partial}{\partial t_\dt}$, hence
\begin{align}\label{eq rcv2}
\rcv=\Ra+t_2\left( [\Ra,\frac{\partial}{\partial t_2}]-\rvf\right)
&=\Ra+t_2\left( [\Ra,\cvf+t_{\dt-2}\frac{\partial}{\partial t_\dt}]-\rvf\right)\\
&=\Ra+t_2t_{\dt-2}\frac{\partial}{\partial t_{\dt-2}}-t_2t_{\dt}\frac{\partial}{\partial t_\dt}\,.\nonumber
\end{align}
Note that in the last equality of the above equation we used \eqref{eq 28/11/2018} and the fact that $[\Ra,\cvf]=\rvf$. Thus,
\begin{align}\label{}
  [\rcv,\frac{\partial}{\partial t_2}] & = [\Ra,\frac{\partial}{\partial t_2}]+[t_2t_{\dt-2}\frac{\partial}{\partial t_{\dt-2}},\frac{\partial}{\partial t_2}]-[t_2t_{\dt}\frac{\partial}{\partial t_d},\frac{\partial}{\partial t_2}]\nonumber\\
   &=  [\Ra,\frac{\partial}{\partial t_2}]-\frac{\partial}{\partial t_2}(t_2t_{\dt-2})\frac{\partial}{\partial t_{\dt-2}}+
   \frac{\partial}{\partial t_2}(t_2t_{\dt})\frac{\partial}{\partial t_{\dt}}\nonumber\\
   &= [\Ra,\frac{\partial}{\partial t_2}]-t_{\dt-2}\frac{\partial}{\partial t_{\dt-2}}+
   t_{\dt}\frac{\partial}{\partial t_{\dt}}\mathop  = \limits^{\eqref{eq 28/11/2018}}[\Ra,\frac{\partial}{\partial t_2}]-[\Ra,t_{\dt-2}\frac{\partial}{\partial t_\dt}]\nonumber\\
   &=[\Ra,\frac{\partial}{\partial t_2}-t_{\dt-2}\frac{\partial}{\partial t_\dt}]\mathop = \limits^{\eqref{eq cvf odd}} [\Ra,\cvf]=\rvf\,.\nonumber
\end{align}
We know that $\Ra$ is quasi-homogeneous of degree $2$ and $\deg(t_2)=2$, hence \eqref{eq rcv2} implies that $\rcv$ is quasi-homogeneous of degree $2$. In order to get $\rcv\in \amsy$, first observe that $\frac{\partial}{\partial t_\dt}=\Ra_{\gL_{1(n+1)}}\in \amsy$. Hence,
\[
\frac{\partial}{\partial t_2}=\cvf+t_{\dt-2}\frac{\partial}{\partial t_\dt}\in\amsy\,,
\]
which yields $\rcv\in\amsy$, and the proof is complete.
\end{proof}

\begin{coro}\label{cor sl2c D}
The Lie algebra generated by the vector fields $\rcv$, $\rvf$ and $(1+\delta_2^n)\frac{\partial}{\partial t_2}$ in the AMSY-Lie algebra $\amsy\subset \mathfrak{X}(\T)$ is isomorphic to $\sl2$.
\end{coro}
\begin{proof}
  It suffices to show that $[\rcv,\frac{\partial}{\partial t_2}]=\rvf,\ [\rvf,\rcv]=2\rcv,\ [\rvf,\frac{\partial}{\partial t_2}]=-2\frac{\partial}{\partial t_2}$. The truth of the first bracket is guaranteed by Lemma \ref{lemm fund}, and the last bracket follows from a simple computation after using \eqref{eq rvf even} or \eqref{eq rvf odd}, and \eqref{eq lie bracket in a chart}. To demonstrate the second bracket $[\rvf,\rcv]=2\rcv$, the same argument given in the proof of Lemma \ref{lemm fund} works perfectly when $\rcv$ is the vector field given in {\rm {\bf Option 1}}, and when $\rcv$ is the one defined in {\rm {\bf Option 2}} for the cases $n=1,2,3,4$ or for even integers $n\geq 5$. For odd integers $n\geq 5$, we first use \eqref{eq rcv2} to obtain:
  \[
  [\rvf,\rcv]=[\rvf,\Ra]+[\rvf,t_2t_{\dt-2}\frac{\partial}{\partial t_{\dt-2}}-t_2t_{\dt}\frac{\partial}{\partial t_\dt}].
  \]
  Then the statement follows from the fact $[\rvf,\Ra]=2\Ra$ given in Theorem \ref{theo 3} and using \eqref{eq lie bracket in a chart} for $\rvf$ stated in \eqref{eq rvf odd}.
\end{proof}

\section{Rankin-Cohen algebras for \cy (quasi-)modular forms}\label{section RCACY}
Let us suppose that $\t_1,\t_2,\ldots,\t_{\dt}$ denote the components of a solution of the vector field $\rcv$, where $\rcv$ is the vector field given either in {\rm {\bf Option 1}} or in {\rm {\bf Option 2}} of Section \ref{subsection rcv}. The reader should take care to differ the notations
$\t_1,\t_2,\ldots,\t_{\dt}$ which stand for solution components of $\rcv$ from the notations $t_1,t_2,\ldots,t_{\dt}$ which are used for the coordinate charts  of $\T$. Nevertheless, any solution component $\t_i$ is associated with the coordinate $t_i$. We define the \emph{space of \cy quasi-modular forms} $\cyqmfs$ and  the \emph{space of \cy modular forms} $\cymfs$, respectively, as follows:
\begin{align}\label{}
  \cyqmfs &:=\C[\t_1,\t_2,\t_3,\ldots,\t_{\dt},\frac{1}{\t_{n+2}(\t_{n+2}-\t_1^{n+2})\check{\t}}]\,, \label{eq cyqmfs} \\
  \cymfs & :=\C[\t_1,\widehat{\ \t_2},\t_3,\t_4,\ldots,\t_{\dt},\frac{1}{\t_{n+2}(\t_{n+2}-\t_1^{n+2})\check{\t}}]\,, \label{eq cymfs}
\end{align}
in which $\check{\t}$ is associated with $\check{t}$ given in \eqref{eq ems} or \eqref{eq rfs} and the symbol $\widehat{\ \t_2}$ means that the component $\t_2$ is omitted, i.e.,  $\t_2\notin \cymfs$. Indeed, we have $\cyqmfs=\cymfs[\t_2]$ and in our generalization the \cy quasi-modular form $\t_2$ has the role of the quasi-modular form $E_2$ in the theory of quasi-modular forms.  Let us attach to any solution component $\t_i$, $1\leq i \leq \dt$, the weight $\deg(\t_i)=w_i$, in which the non-negative integers $w_i$'s are given in \eqref{eq rvf gf}. For any integer $r\in \Z$ we define $\cyqmfs_r$ and $\cymfs_r$ to be the $\C$-vector spaces generated by  $\{\,f\in \cyqmfs\, | \, \deg(f)=r\,\}$ and $\{\,f\in \cymfs\, | \, \deg(f)=r\,\}$, respectively. Note that any constant in $\C$ is considered as a weight zero \cy (quasi-)modular form. Therefore, elements of $\cyqmfs_r$ and $\cymfs_r$ are \cy quasi-modular forms and \cy modular forms of weight $r$, respectively. In particular, $\t_2$ is a \cy quasi-modular form of weight $2$, see Remark \ref{rem weights}, and the other $\t_j$'s, $1\leq j\leq \dt$ and $j\neq2$, are \cy modular forms of weight $w_j$. In particular we have:
\begin{equation}\label{eq cyqmfs and cymfs}
  \cyqmfs=\bigoplus_{r\in \Z}\cyqmfs_r\ \ \ \textrm{and} \ \ \ \cymfs=\bigoplus_{r\in \Z}\cymfs_r\,.
\end{equation}
Thus, $\cyqmfs$ and $\cymfs$ are commutative and associative graded algebras on $\C$.

\begin{nota}
From now on $\mdo, \, \rdo$ and $\cdo$ refer to the differential operators on $\cyqmfs$ induced by the vector fields $\Ra,\, \rvf$ and $\cvf$, respectively, in which we substitute the coordinate chart $t_j$, $1\leq j \leq \dt$, by the solution component $\t_j$ and $\frac{\partial}{\partial t_j}$ by the partial derivation $\frac{\partial}{\partial \t_j}$. For example, if $\Ra=\sum_{j=1}^{\dt}\Ra^j(t_1,t_2,\ldots,t_d)\frac{\partial}{\partial t_j}$, with $\Ra^j(t_1,t_2,\ldots,t_d)\in \O_\T$, then $\mdo=\sum_{j=1}^{\dt}\Ra^j(\t_1,\t_2,\ldots,\t_d)\frac{\partial}{\partial \t_j}$. We consider the Lie bracket of the such obtained differential operators the same as the Lie bracket of the associated vector fields. Hence, due to Theorem \ref{theo 3} we get:
\[
[\mdo,\cdo]=\rdo \ , \ \ [\rdo,\mdo]=2\mdo \ , \ \ [\rdo,\cdo]=-2\cdo
\, .
\]
\end{nota}
We recall that, for an integer $d$, a degree $d$ differential operator $D$ on $\cyqmfs$, denoted  by $D:\cyqmfs_\ast\to \cyqmfs_{\ast+d}$, is a differential operator that satisfies $D(\cyqmfs_r)\subseteq \cyqmfs_{r+d}$ for any positive integer $r$. Indeed, if we can write $D=\sum_{j=1}^{\dt}D^j\frac{\partial}{\partial \t_j}$, with $D^j\in \cyqmfs$, then $D$ has degree $d$ provided $\deg(D^j)-w_j=d$ for any $1\leq j\leq \dt$. A degree $d$ differential operator on $\cymfs$ is defined analogously.

\begin{defi}
We define the derivation $\rcdo$ on $\cyqmfs$ to be the differential operator induced by the vector field $\rcv$. In fact, $\rcdo$ is as follows:
\begin{equation}\label{eq rcdo}
  \rcdo:=\left\{
           \begin{array}{ll}
             \mdo+\t_2([\mdo,(1+\delta_2^n)\frac{\partial}{\partial \t_2}]-\rdo), & \hbox{if $\rcv \in$ {\bf Option 2}, and $n\geq 3$ is odd;} \\
             \mdo, & \hbox{otherwise;}
           \end{array}
         \right. ,
\end{equation}
By the \emph{\rs} $\rsdo$ on $\cyqmfs$ we mean the differential operator that on the generators of $\cyqmfs$  is defined as follows:
\begin{align}
   & \rsdo f:=\rcdo f+(1-\frac{1}{2}\delta_2^n)r\t_2 f\,, \ \  \forall f\in \cyqmfs_r \ {\rm and} \ \forall r\in \Z. \label{eq rsdo}
\end{align}
\end{defi}

We would like that the derivation $\rcdo$ and the \rs $\rsdo$  behave the
same as the usual derivation \eqref{eq qmf der} and the Ramanujan-Serre derivation \eqref{eq rsd} of the classical quasi-modular form theory, respectively.
In the following example we state the derivations $\rcdo$ and $\rsdo$ explicitly for $n=1,2,3,4$.

\begin{exam}\label{exam 1} {\rm
In \cite{younes2} we found $\Ra, \rvf,\cvf$ explicitly for $n=1,2,3,4$. In these cases, we obtain the derivation $\rcdo$ and the \rs $\rsdo$ as follows:
\begin{itemize}
  \item {\bf $n=1$.}
  \begin{align}
&\Ra=(-t_1t_2-9(t_1^3-t_3))\frac{\partial}{\partial
t_1}+(81t_1(t_1^3-t_3)-t_2^2)\frac{\partial}{\partial
t_2}+(-3t_2t_3)\frac{\partial}{\partial t_3}\,, \\
 &\rvf=t_1\frac{\partial}{\partial
t_1}+2t_2\frac{\partial}{\partial
t_2}+3t_3\frac{\partial}{\partial t_3},\label{eq rvf n=1}\\
&\cvf=\frac{\partial}{\partial t_2}\, . \label{eq cvf1}
  \end{align}
  By definition, the vector field \eqref{eq rvf n=1} implies $\deg(\t_1)=1,\, \deg(\t_2)=2$ and $\deg(\t_3)=3$. Since $[\Ra,\cvf]=\rvf$, we observe that:
  \begin{align}
    \rcdo & =\mdo\,, \\
    \rsdo & =-9(\t_1^3-\t_3)\frac{\partial }{\partial \t_1}+(81\t_1(\t_1^3-\t_3)+\t_2^2)\frac{\partial}{\partial
\t_2}\,.
  \end{align}
If we let $\rsdo$ acts just on $\cymfs$, then we get:
\[
 \rsdo  =-9(\t_1^3-\t_3)\frac{\partial }{\partial \t_1}\,.
\]
  \item {\bf $n=2$.}
\begin{align}
\Ra&=(t_3-t_1t_2)\frac{\partial}{\partial
t_1}+(2t_1^2-\frac{1}{2}t_2^2)\frac{\partial}{\partial
t_2}+(-2t_2t_3+8t_1^3)\frac{\partial}{\partial t_3}+(-4t_2t_4)\frac{\partial}{\partial t_4}\,,\\
\rvf&=2t_1\frac{\partial}{\partial
t_1}+2t_2\frac{\partial}{\partial
t_2}+4t_3\frac{\partial}{\partial t_3}+8t_4\frac{\partial}{\partial t_4},\label{eq rvf n=2}\\
\cvf&=2\frac{\partial}{\partial t_2}\,,
\end{align}
where the polynomial equation $t_3^2=4(t_1^4-t_4)$ holds
among $t_i$'s. From \eqref{eq rvf n=2} we get $\deg(\t_1)=2, \, \deg(\t_2)=2, \, \deg(\t_3)=4$ and $\deg(\t_4)=8$. Hence, due to \eqref{eq rcdo} and \eqref{eq rsdo} we find:
  \begin{align}
    \rcdo & =\mdo\,, \\
    \rsdo & =\t_3\frac{\partial }{\partial \t_1}+(2\t_1^2+\frac{1}{2}\t_2^2)\frac{\partial}{\partial
\t_2}+8\t_1^3\frac{\partial }{\partial \t_3}\,.
  \end{align}
  In the case that $\rsdo$ is considered on $\cymfs$ we have:
  \[
  \rsdo =\t_3\frac{\partial }{\partial \t_1}+8\t_1^3\frac{\partial }{\partial \t_3}\,.
  \]
  \item {\bf $n=3$.}
\begin{align}
\Ra&=(t_3-t_1t_2)\frac{\partial}{\partial
t_1}+\frac{t_3^3t_4-5^4t_2^2(t_1^5-t_5)}{5^4(t_1^5-t_5)}\,\frac{\partial}{\partial t_2} \label{eq mvf R3}\\
&+\frac{t_3^3t_6-3\times5^4t_2t_3(t_1^5-t_5)}{5^4(t_1^5-t_5)}\,\frac{\partial}{\partial
t_3}  +(-t_2t_4-t_7)\frac{\partial}{\partial t_4}\nonumber\\
&+(-5t_2t_5)\frac{\partial}{\partial
t_5}+(5^5t_1^3-t_2t_6-2t_3t_4)\frac{\partial}{\partial
t_6}+(-5^4t_1t_3-t_2t_7)\frac{\partial}{\partial t_7}\,, \nonumber\\
\rvf&=t_1\frac{\partial}{\partial t_1}+2t_2\frac{\partial}{\partial
t_2}+3t_3\frac{\partial}{\partial t_3}+5t_5\frac{\partial}{\partial
t_5}+t_6\frac{\partial}{\partial t_6}+2t_7\frac{\partial}{\partial t_7}\,, \\
\cvf&=\frac{\partial}{\partial t_2}-t_4\frac{\partial}{\partial
t_7}\,.
\end{align}
We obtain $\deg(\t_1)=1, \ \deg(\t_2)=2, \ \deg(\t_3)=3, \ \deg(\t_4)=0, \ \deg(\t_5)=5, \ \deg(\t_6)=1, \ \deg(\t_7)=2$.
For $\rcv$ given in {\bf Option 1} remember that we substitute the coordinate $t_7$ by:
 \[
  \tilde{t}_7:=t_7+t_2t_4.
 \]
 from which we obtain:
 \begin{align*}
  &\Ra(\tilde{t}_7)=-5^4t_1t_3+\frac{t_3^3t_4^2}{5^4(t_1^5-t_5)}-2t_2\tilde{t}_7.
 \end{align*}
Note that in this case $\t_7$ is the component of a solution of $\rcv$ associated with coordinate $\tilde t_7$.
Hence, we get the derivation $\rcdo$ on $\cyqmfs$ as follows:
\begin{align*}
\rcdo=\mdo&=\Big(\t_3-\t_2\t_1\Big)\frac{\partial}{\partial
\t_1}+\left(\frac{\t_3^3\t_4}{5^4(\t_1^5-\t_5)}-\t_2^2\right)\,\frac{\partial}{\partial \t_2} \label{eq mvf R3}\\
&+\left(\frac{\t_3^3\t_6}{5^4(\t_1^5-\t_5)}-3\t_2\t_3 \right)\,\frac{\partial}{\partial
\t_3}  -{\t}_7\,\frac{\partial}{\partial \t_4}-5\t_2\t_5\,\frac{\partial}{\partial \t_5}\nonumber\\
&+\Big(5^5\t_1^3-2\t_3\t_4-\t_2\t_6\Big)\frac{\partial}{\partial
\t_6}+\Big(-5^4\t_1\t_3+\frac{\t_3^3\t_4^2}{5^4(\t_1^5-\t_5)}-2\t_2{\t}_7\Big)\frac{\partial}{\partial {\t}_7}\,, \nonumber
\end{align*}
and we obtain $\rsdo$ on $\cymfs$ as follows:
\begin{equation*}
   \rsdo =\t_3\frac{\partial}{\partial \t_1}+\frac{\t_3^3\t_6}{5^4(\t_1^5-\t_5)}\,\frac{\partial}{\partial
\t_3} -\t_7\frac{\partial}{\partial \t_4}+(5^5\t_1^3-2\t_3\t_4)\frac{\partial}{\partial
\t_6}-\left(5^4\t_1\t_3-\frac{\t_3^3\t_4^2}{5^4(\t_1^5-\t_5)}\right)\frac{\partial}{\partial \t_7} \,.
\end{equation*}
If $\rcv$ is the vector field given in {\bf option 2}, then we get $\rcdo:\cyqmfs\to \cyqmfs$  as follows:
\begin{align*}
\rcdo&=\Big(\t_3-\t_2\t_1\Big)\frac{\partial}{\partial
\t_1}+\left(\frac{\t_3^3\t_4}{5^4(\t_1^5-\t_5)}-\t_2^2\right)\,\frac{\partial}{\partial \t_2} \label{eq mvf R3}\\
&+\left(\frac{\t_3^3\t_6}{5^4(\t_1^5-\t_5)}-3\t_2\t_3 \right)\,\frac{\partial}{\partial
\t_3}  -{\t}_7\,\frac{\partial}{\partial \t_4}-5\t_2\t_5\,\frac{\partial}{\partial \t_5}\nonumber\\
&+\Big(5^5\t_1^3-2\t_3\t_4-\t_2\t_6\Big)\frac{\partial}{\partial
\t_6}+\Big(-5^4\t_1\t_3-2\t_2{\t}_7\Big)\frac{\partial}{\partial {\t}_7}\,, \nonumber
\end{align*}
and we obtain $\rsdo:\cymfs\to \cymfs$ as follows:
\begin{equation}\label{}
   \rsdo =\t_3\frac{\partial}{\partial \t_1}+\frac{\t_3^3\t_6}{5^4(\t_1^5-\t_5)}\,\frac{\partial}{\partial
\t_3} -\t_7\frac{\partial}{\partial \t_4}+(5^5\t_1^3-2\t_3\t_4)\frac{\partial}{\partial
\t_6}-5^4\t_1\t_3\frac{\partial}{\partial \t_7} \,.
\end{equation}

  \item {\bf $n=4$.}
\begin{align}
\Ra&=(t_3-t_1t_2)\frac{\partial}{\partial
t_1}+\frac{6^{-2}t_3^2t_4t_8-t_1^6t_2^2+t_2^2t_6}{t_1^6-t_6}\,\frac{\partial}{\partial
t_2} \label{eq mvf4} \\ &+\frac{6^{-2}t_3^2t_5t_8-3t_1^6t_2t_3+3t_2t_3t_6}{t_1^6-t_6}\,\frac{\partial}{\partial
t_3}
+\frac{-6^{-2}t_3^2t_7t_8-t_1^6t_2t_4+t_2t_4t_6}{t_1^6-t_6}\,\frac{\partial}{\partial
t_4} \nonumber \\ &+\frac{6^{-2}t_3t_5^2t_8-4t_1^6t_2t_5-2t_1^6t_3t_4+5t_1^4t_3t_8+4t_2t_5t_6+2t_3t_4t_6}{2(t_1^6-t_6)}\,\frac{\partial}{\partial
t_5} \nonumber
\\&+(-6t_2t_6)\,\frac{\partial}{\partial
t_6}+\frac{6^{-2}t_4^2-t_1^2}{2\times
6^{-2}}\,\frac{\partial}{\partial
t_7}+\frac{-3t_1^6t_2t_8+3t_1^5t_3t_8+3t_2t_6t_8}{t_1^6-t_6}\,\frac{\partial}{\partial
t_8}\,,
 \nonumber\\
\rvf&=t_1\frac{\partial}{\partial t_1}+2t_2\frac{\partial}{\partial
t_2}+3t_3\frac{\partial}{\partial t_3}+t_4\frac{\partial}{\partial
t_4}+2t_5\frac{\partial}{\partial
t_5}+6t_6\frac{\partial}{\partial t_6}+3t_8\frac{\partial}{\partial t_8}\,, \label{eq rvf4}\\
\cvf&=\frac{\partial}{\partial t_2}\,, \label{eq cvf4}
\end{align}
where the equation $t_8^2=36(t_1^6-t_6)$ holds among $t_i$'s.  Analogous to the pervious cases we have $\deg(\t_1)=1, \ \deg(\t_2)=2, \ \deg(\t_3)=3, \ \deg(\t_4)=1, \ \deg(\t_5)=2, \ \deg(\t_6)=6, \ \deg(\t_7)=0, \ \deg(\t_8)=3$. Due to \eqref{eq rcdo} we find:
\begin{equation}\label{}
 \rcdo = \mdo\,
 \end{equation}
and \eqref{eq rsdo} yields the \rs on $\cymfs$ as follows:
\begin{align}
\rsdo &=\t_3\frac{\partial}{\partial \t_1} +\frac{6^{-2}\t_3^2\t_5t_8}{\t_1^6-\t_6}\,\frac{\partial }{\partial\t_3}
-\frac{6^{-2}\t_3^2\t_7\t_8}{\t_1^6-\t_6}\,\frac{\partial }{\partial
\t_4}  \\ &+\frac{6^{-2}\t_3\t_5^2\t_8-2\t_1^6\t_3\t_4+5\t_1^4\t_3\t_8+2\t_3\t_4\t_6}{2(\t_1^6-\t_6)}\,\frac{\partial}{\partial
\t_5}\nonumber \\ &+\frac{6^{-2}\t_4^2-\t_1^2}{2\times
6^{-2}}\,\frac{\partial }{\partial\t_7}+\frac{3\t_1^5\t_3\t_8}{\t_1^6-\t_6}\,\frac{\partial }{\partial \t_8}\,.\nonumber
\end{align}
\end{itemize}}
\end{exam}

\begin{rem}\label{rem pumbi1}
\begin{enumerate}
  \item If we look closely to all cases stated in Example \ref{exam 1} we find out that the derivation $\rcdo$ and the \rs $\rsdo$ have degree $2$. Besides these, the \rs $\rsdo$ sends any element of $\cymfs$ to another element of $\cymfs$. More precisely, the same as what we mentioned for the Ramanujan-Serre derivation given in \eqref{eq rsd2}, in all the above cases we observe that for any $f\in\cymfs_r$ the term $(1-\frac{1}{2}\delta_2^n)r\t_2f$ in \eqref{eq rsdo}  kills all the terms including $\t_2$ in $\rcdo f$ which implies $\rsdo f\in \cymfs_{r+2}$, and consequently $\cymfs$ is closed under $\rsdo$.
      All these facts hold for any positive integer $n$ which are stated in Theorem \ref{theo 4}.
  \item In Example \ref{exam 1} we stated the derivation $\rcdo$ explicitly in the cases $n=1,2,3,4$. For $n\geq 5$, due to the proof of Lemma \ref{lemm fund}, we can state $\rcdo$ explicitly as follows:
       \begin{itemize}
         \item if $\rcv$ is the vector field given in {\bf Option 2} and $n\geq 5$ is odd, then $\rcdo=\mdo+\t_2\t_{\dt-2}\frac{\partial}{\partial \t_{\dt-2}}-\t_2\t_\dt\frac{\partial}{\partial \t_{\dt}}$\,,
         \item otherwise, we have $\rcdo=\mdo$.
       \end{itemize}
\end{enumerate}
\end{rem}

Now we are in the situation that we can present the proof of Theorem \ref{theo 4}.\\

{\bf Proof of Theorem \ref{theo 4}.}
\begin{enumerate}
  \item Due to Lemma \ref{lemm fund} the proof is straightforward, since the differential operator $\rcdo$ is induced by the vector field $\rcv$ which is a quasi-homogeneous vector field of degree $2$.
  \item First note that according to Remark \ref{rem weights} we always have $\deg(\t_2)=w_2=2$. Hence, from part 1 and \eqref{eq rsdo} we deduce that $\rsdo$ is a degree $2$ differential operator. To prove that for all $f\in \cymfs$ we get $\rsdo f\in \cymfs$, it is enough to observe that for all integers $r$ and for all $f\in \cymfs_r$ we have $\partial f\in \cymfs_{r+2}$, which is equivalent to:
\begin{align*}
    \partial \t_j \in \cymfs_{w_j+2}\,, \  \ \forall j\neq 2 ,
   & \Leftrightarrow (1+\delta_2^n)\frac{\partial}{\partial \t_2}(\partial \t_j)=0\,, \  \ \forall j\neq 2 , \\
   & \Leftrightarrow (1+\delta_2^n)\frac{\partial}{\partial \t_2}(\rcdo \t_j+(1-\frac{\delta_2^n}{2})w_j\t_2\t_j)=0\,, \  \ \forall j\neq 2 ,\\
   & \Leftrightarrow (1+\delta_2^n)\frac{\partial}{\partial \t_2}(\rcdo \t_j)=-w_j\t_j\,, \  \ \forall j\neq 2 ,\\
   & \Leftarrow \sum_{j=1}^{\dt}(1+\delta_2^n)\frac{\partial}{\partial \t_2}(\rcdo \t_j)\frac{\partial}{\partial \t_j}=- \sum_{j=1}^{\dt}w_j\t_j \frac{\partial}{\partial \t_j}=-\rdo\,, \\
   & \Leftrightarrow [(1+\delta_2^n)\frac{\partial}{\partial \t_2},\rcdo]=-\rdo \, ,\\
   & \Leftrightarrow [\rcdo,(1+\delta_2^n)\frac{\partial}{\partial \t_2}]=\rdo \,.
\end{align*}
The last affirmation is valid due to Lemma \ref{lemm fund}, which completes the proof. \hfill\(\square\)\\
\end{enumerate}

Next, to use Proposition \ref{prop crca}, we need the \cy quasi-modular forms of positive weight. Hence, we consider the spaces of \cy quasi-modular forms  $\cyqmfs^{>0}$ and \cy modular forms $\cymfs^{>0}$ of positive weight as follows:
\begin{align}
  \cyqmfs^{>0}:=\bigoplus_{r\geq 0}\cyqmfs_r \qquad , \qquad\cymfs^{>0}:=\bigoplus_{r\geq 0}\cymfs_r\,, \label{eq cymfs pw}
\end{align}
in which we suppose that $\cyqmfs_0=\cymfs_0=\C$.
Thus, the space of \cy quasi-modular forms of positive weight $\cyqmfs^{>0}$ is a commutative and associative graded algebra with unit over the field $\C$ together with the derivation $\rcdo: \cyqmfs^{>0}_\ast \to \cyqmfs^{>0}_{\ast+2}$ of degree 2. Therefore, due to Remark \ref{rem srca}, $(\cyqmfs^{>0},[\cdot,\cdot]_{\rcdo,\ast})$ is a standard Rankin-Cohen, and hence a Rankin-Cohen algebra. We call $[\cdot,\cdot]_{\rcdo,\ast}$ the \emph{Rankin-Cohen bracket for \cy quasi-modular forms}, and for any non-negative integers $k,r,s$ it is defined as
\begin{equation}\label{eq rcb r1}
  [f,g]_{\rcdo,k}:=\sum_{i+j=k}(-1)^j\binom{k+r-1}{i}\binom{k+s-1}{j}f^{(j)}g^{(i)},\ \forall f\in \cyqmfs_r \,, \   \forall g\in \cyqmfs_s\,,
\end{equation}
where $f^{(j)}=\rcdo^jf$ and $g^{(j)}=\rcdo^jg$ refer to the $j$-th derivative of $f$ and $g$ under $\rcdo$, respectively. It is evident that $[f,g]_{\rcdo,k}\in \cyqmfs_{r+s+2k}$.
Next, we demonstrate Theorem \ref{theo 5} which shows that the space of \cy modular forms of positive weight $\cymfs^{>0}$ is closed under the Rankin-Cohen bracket for \cy quasi-modular forms given
in \eqref{eq rcb r1}.
\bigskip\\
{\bf Proof of Theorem \ref{theo 5}.}\\
The idea of the proof is to use Proposition \ref{prop crca} and its proof. To this end, first note that according to the part 2 of Theorem \ref{theo 4} the \rs $\partial:\cymfs^{>0}_\ast\to \cymfs^{>0}_{\ast+2}$ is a degree 2 differential operator. If we set
$\Lambda=\Lambda(\t_1,\t_2,\ldots,\t_\dt)$, where $\Lambda$ is given in Lemma \ref{lemm R^2}, then the same lemma yields $\Lambda\in\cymfs_4$. Therefore, from Proposition \ref{prop crca} we get that $(\cymfs^{>0},[\cdot,\cdot]_{\partial,\Lambda,\ast})$, where the $k$-th bracket $[\cdot,\cdot]_{\partial,\Lambda,k}\,,\ k\geq 0,$ is given by \eqref{eq crcb}, is a canonical Rankin-Cohen algebra. On the other hand, by letting
  $
\lambda=(\frac{1}{2}\delta_2^n-1)\t_2\,,
$
from \eqref{eq rsdo} we obtain
\begin{equation}\label{eq pumbi2}
\rcdo f=\rsdo f+r\lambda f\,,\ \forall f\in \cymfs_r\,.
\end{equation}
Furthermore, if we write $\rcdo=\sum_{j=1}^{\dt}\rcv^j\frac{\partial}{\partial \t_j}$, with $\rcv^j\in \cyqmfs$, then
\begin{equation}\label{eq Dlambda}
\rcdo(\lambda)=(\frac{1}{2}\delta_2^n-1)\rcdo(\t_2)=(\frac{1}{2}\delta_2^n-1)\rcv^2\,.
\end{equation}
Considering $\mdo=\sum_{j=1}^{\dt}\Ra^j\frac{\partial}{\partial \t_j}$, with $\Ra^j\in \cyqmfs$, the part 2 of Remark \ref{rem pumbi1} yields $\rcv^2=\Ra^2$. This fact along with \eqref{eq Dlambda} and \eqref{eq Lambda} implies:
\begin{equation}\label{eq Dlambda2}
\rcdo(\lambda)=\Lambda+\lambda^2\,.
\end{equation}
The relations \eqref{eq pumbi2} and \eqref{eq Dlambda2} show that \eqref{eq der zag} is satisfied. Hence, from the proof of Proposition \ref{prop crca} we obtain
$[\cdot,\cdot]_{\partial,\Lambda,\ast}=[\cdot,\cdot]_{\rcdo,\ast}$ (see \eqref{eq []=[]}). Finally, since $\cymfs^{>0}$ is closed
under $[\cdot,\cdot]_{\partial,\Lambda,\ast}$, we conclude that $\cymfs^{>0}$ is closed under $[\cdot,\cdot]_{\rcdo,\ast}$, and this finishes the proof of the theorem.\hfill\(\square\)\\

In particular, Theorem \ref{theo 5} implies that $(\cymfs^{>0},[\cdot,\cdot]_{\rcdo,\ast})$ is a sub Rankin-Cohen algebra of $(\cyqmfs^{>0},[\cdot,\cdot]_{\rcdo,\ast})$.

\begin{coro}
The \rc bracket for \cy quasi-modular forms $[\cdot , \cdot]_{\rcdo,\ast}$  endows $\cymfs^{>0}$ with a canonical \rc algebra structure.
\end{coro}

\subsection{Examples of Rankin-Cohen brackets of \cy modular forms}
We know that the modular discriminant is given by $\Delta=\frac{1}{1728}(E_4^3-E_6^2)$, which is related with the discriminant $t_2^3-27t_3^2$ of the family of elliptic curves stated in \eqref{eq ufec}. One can easily compute (or find in \cite{zag94}) the following examples of Rankin-Cohen brackets \eqref{eq rcb} of modular forms:
\begin{align}\label{eq ex rcb mf}
   & [E_4,E_6]_1=-3456 \Delta\,, \ \ [E_4,E_6]_2=0 \,, \ \ [E_4,E_4]_2=4800\Delta, \\
   & [E_6,E_6]_2=-21168E_4 \Delta\,, \ \  [\Delta,\Delta]_2=-13E_4\Delta^2\,.\nonumber
\end{align}
Note that for any (quasi-)modular form or any \cy (quasi-)modular form $f$ of non-negative weight $r$ and any integer $k\geq 0$ it is evident by definition that :
\begin{equation}\label{eq [ff]odd}
  [f,f]_{2k+1}=0\ \ {\rm or} \ \ [f,f]_{\rcdo,2k+1}=0 \,.
\end{equation}

For any positive integer $n$, the discriminant of the Dwork family \eqref{eq DF tt} is given by the polynomial $t_{n+2}(t_1^{n+2}-t_{n+2})$. Hence, in the rest of this section for any $n$ we fix the notation $\Delta:=\t_{n+2}(\t_1^{n+2}-\t_{n+2})$. Next, we compute a few examples of Rankin-Cohen brackets \eqref{eq rcb r1} of \cy modular forms for $n=1,2,3,4$, which are motivated by examples given in \eqref{eq ex rcb mf}.
\begin{description}
  \item[$\bullet \ n=1$.] In this case we found $\t_1,\t_2,\t_3$ in the first list of \eqref{eq solution R01} and we have $\Delta=\t_3(\t_1^3-\t_3)$. The Rankin-Cohen brackets are calculated as follows:
      \begin{align}\label{eq ex rcb cymf1}
      & [\t_1,\t_3]_{\rcdo,1}=27\Delta\,, \ \ [\t_1,\t_3]_{\rcdo,2}=729\t_1^2\Delta \,, \ \ [\t_1,\t_1]_{\rcdo,2}=324\Delta, \\
      & [\t_3,\t_3]_{\rcdo,2}=-2916\t_1\Delta\,, \ \  [\Delta,\Delta]_{\rcdo,2}=-5103\t_1^4\Delta^2\,.\nonumber
      \end{align}
      Before passing to the next case, we express the combinations of $\t_1,\t_2,\t_3$ which appeared in the right hand side of the above relations in terms of eta and theta functions that seem to us interesting. These relations are obtained thanks to \cite{Oeis} and one can find out more about them by seeing the corresponding pages and references given there.  By comparing the coefficients of $\t_1$ with \cite[A004016]{Oeis} we  find:
      \begin{equation}\label{eq t1 2}
        \t_1=\frac{1}{3}(\theta_3(q)\theta_3(q^3) + \theta_2(q)\theta_2(q^3)),
      \end{equation}
      and for $\t_1^2$ and $\t_1^4$ the reader is referred to \cite[A008653]{Oeis} and  \cite[A008655]{Oeis}, respectively. After computing the $q$-expansion of $\Delta$, from \cite[A007332]{Oeis} we get:
      \begin{equation}\label{eq Delta1}
        \Delta=\frac{1}{27}\eta^6(q)\eta^6(q^3)\,,
      \end{equation}
      and on account of \cite[A136747]{Oeis} we get:
      \begin{equation}\label{eq t1Delta1}
       \t_1^2 \Delta=\frac{1}{243}\eta^6(q)\eta^4(q^3)\left(\eta^3(q) + 9 \eta^3(q^9)\right)^2.
      \end{equation}
      The equations \eqref{eq t1 2}, \eqref{eq Delta1} and \eqref{eq t1Delta1} yield:
      \begin{equation}\label{eq t1 3}
        3\t_1=\theta_3(q)\theta_3(q^3) + \theta_2(q)\theta_2(q^3)=\frac{\eta^3(q) + 9 \eta^3(q^9)}{\eta(q^3)}.
      \end{equation}

  \item[$\bullet \ n=2$.] Here $\t_1,\t_2,\t_4$ are stated in the second list of \eqref{eq solution R01}. We know that $\Delta=\t_4(\t_1^4-\t_4)$, and we obtain:
      \begin{align}\label{eq ex rcb cymf2}
      & [\t_1,\t_4]_{\rcdo,1}=-8\t_3\t_4\,, \ \ [\t_1,\t_4]_{\rcdo,2}=192\t_1^3\t_4 \,, \ \ [\t_1,\t_1]_{\rcdo,2}=36\t_1^4-9\t_3^2=36\t_4, \\
      & [\t_4,\t_4]_{\rcdo,2}=-576\t_1^2\t_4^2\,, \ \  [\Delta,\Delta]_{\rcdo,2}=-1088\t_1^2\t_4(\t_1^4+8\t_4)\Delta\,.\nonumber
      \end{align}
      Note that in the third bracket of \eqref{eq ex rcb cymf2} we used the fact that $\t_3^2=4(\t_1^4-\t_4)$, which also implies:
      \begin{equation}\label{eq [t1t4]2}
        [\t_1,\t_4]_{\rcdo,1}^2=64\t_3^2\t_4^2=256\t_4\Delta\,.
      \end{equation}
  \item[$\bullet \ n=3$.] If $\rcv$ is the vector field given in the {\bf Option 1}, then one can find the $q$-expansion of $\t_1,\t_2,\ldots,\t_7$ in \cite{ho22}. For $\rcv$ given in {\bf Option 2} we can not find $q$-expansion of solution components of $\rcv$ around $\infty$, because if we suppose that $\dot t_j=aq\frac{\partial t_j}{\partial q}$, for a constant $a\in \C$, then we find $a=0$. In this case we have $\Delta=\t_5(\t_1^5-\t_5)$, and we calculate the Rankin-Cohen brackets, for $\rcdo$ induced by $\rcv$ given in both {\bf Option 1} and {\bf Option 2}, as follows:
      \begin{align}\label{eq ex rcb cymf3}
      & [\t_1,\t_5]_{\rcdo,1}=-5\t_3\t_5\,, \ \ [\t_1,\t_5]_{\rcdo,2}=\frac{-4\t_1\t_3^3\t_4\t_5+3\t_3^3\t_5\t_6}{125(\t_1^5-\t_5)} \,, \\
      & [\t_1,\t_1]_{\rcdo,2}=\frac{-2500\t_3^2(\t_1^5-\t_5)-2\t_1\t_3^3(\t_1\t_4-\t_6)}{625(\t_1^5-\t_5)}\,, \ \  [\t_5,\t_5]_{\rcdo,2}=\frac{-6\t_3^3\t_4\t_5^2}{25(\t_1^5-\t_5)}\,,\nonumber\\
      & [\Delta,\Delta]_{\rcdo,2}=\frac{\t_3^2\t_5^2}{25}(\t_1^3(-20625\t_1^5-55000\t_5+22\t_1\t_3\t_6)-44\t_3\t_4(\t_1^5-\t_5)) \,.\nonumber
      \end{align}
  \item[$\bullet \ n=4$.] Here, the first $7$ coefficients of the $q$-expansions of $\t_1,\t_2,\ldots,\t_7,\t_8$ are given in \cite[Table 2]{movnik}. We get $\Delta=\t_6(\t_1^6-\t_6)$ and hence:
      \begin{align}\label{eq ex rcb cymf4}
      & [\t_1,\t_6]_{\rcdo,1}=-6\t_3\t_6\,, \ \ [\t_1,\t_6]_{\rcdo,2}=\frac{-9\t_1\t_3^2\t_4\t_6\t_8+7\t_3^2\t_5\t_6\t_8}{12(\t_1^6-\t_6)} \,, \\
      & [\t_1,\t_1]_{\rcdo,2}=\frac{-72\t_3^2(\t_1^6-\t_6)-\t_1\t_3^2\t_8(\t_1\t_4-\t_5)}{18(\t_1^6-\t_6)}\,, \ \  [\t_6,\t_6]_{\rcdo,2}=\frac{-7\t_3^2\t_4\t_6^2\t_8}{\t_1^6-\t_6}\,,\nonumber\\
      & [\Delta,\Delta]_{\rcdo,2}=\t_3^2\t_6^2(\t_1^4(-1404\t_1^6-4680\t_6+26\t_1\t_5\t_8)-52\t_4\t_8(\t_1^6-\t_6)) \,.\nonumber
      \end{align}
\end{description}

The relations given in \eqref{eq t1dtn+2d} yield $\rcdo \t_1=\t_3-\t_1\t_2$  and $\rcdo \t_{n+2}=-(n+2)\t_2\t_{n+2}$ for any integer $n\geq 3$, from which we conclude the following expected result (see \eqref{eq ex rcb cymf3} and \eqref{eq ex rcb cymf4}):
\begin{equation}\label{eq [t1tn+2]1}
  [\t_1,\t_{n+2}]_{\rcdo,1}=-(n+2)\t_3\t_{n+2}\,, \ \ \forall n\geq 3\,.
\end{equation}
Another interesting point that we observe in the above examples is that in all the cases $n=1,2,3,4$ the bracket $[\Delta,\Delta]_{\rcdo,2}$ is expressed as a polynomial in terms of $\t_1,\t_2,\ldots,\t_\dt$, and we expect that this happens for higher dimensions as well.

It is also worth to point out that for any \cy (quasi-)modular form $f$ of weight $r$, the second Rankin-Cohen bracket $[f,f]_{\rcdo,2}$ provides a second order differential equation which is satisfied by $f$. More precisely, from \eqref{eq rcb r1} we obtain:
\begin{equation}\label{eq [ff]2}
[f,f]_{\rcdo,2}=6f\rcdo^2f-9(\rcdo f)^2,
\end{equation}
which implies that $f$ satisfies the second order ODE:
\begin{equation}\label{eq 2nd order ode}
  6y\rcdo^2y-9(\rcdo y)^2=[f,f]_{\rcdo,2}\,.
\end{equation}
For example, if $n=1$, then from the third bracket of \eqref{eq ex rcb cymf1} we get that the function
\[
\t_1=\frac{1}{3}(2\theta_3(q^2)\theta_3(q^6)-\theta_3(-q^2)\theta_3(-q^6))=\frac{1}{3}(\theta_3(q)\theta_3(q^3) + \theta_2(q)\theta_2(q^3))=\frac{\eta^3(q) + 9 \eta^3(q^9)}{3\eta(q^3)},
\]
satisfies the following second order ODE:
\begin{equation}\label{eq 2nd order ode 1}
  2y\ddot{y}-3\dot{y}^2=4\eta^6(q)\eta^6(q^3)\,,
\end{equation}
in which $\dot{y}=3q\frac{\partial y}{\partial q}=\frac{3}{2\pi i}\frac{d y}{d \tau}$.

\section{Final remarks} \label{section FR}
One of weak points of Theorem \ref{theo 5} is that we are just considering the \cy modular forms of positive weight. If we look closely to the definition of $\cyqmfs$ and $\cymfs$ given in \eqref{eq cyqmfs} and \eqref{eq cymfs}, respectively, we observe that they contain non-constant elements of weight zero and elements of negative weight. For example for $n=3$, the element $\t_4\in \cymfs$ is a non-constant element of weight zero and $\frac{1}{\t_5(\t_1^5-\t_5)}\in \cymfs$ is an element of weight $-10$. Thus, in general it is not necessarily valid that $\cyqmfs_0=\cymfs_0=\C$; indeed, $\cyqmfs_0$ and $\cymfs_0$ are generated by $\C\cup \{\,f\in \cyqmfs\, | \, \deg(f)=0\,\}$ and $\C\cup \{\,f\in \cymfs\, | \, \deg(f)=0\,\}$, respectively.
We can consider the definition of the Rankin-Cohen bracket \eqref{eq rcb r1} for elements of negative weight as well, and hence
we can endow $\cyqmfs$ with a Rankin-Cohen algebra structure.  Using the computer we observed that the Rankin-Cohen brackets of all examined \cy modular forms of negative weight are again \cy modular forms, in the cases $n=1,2,3,4$, but we could not prove theoretically the assertion that the space of \cy modular forms $\cymfs$ is closed under the Rankin-Cohen bracket \eqref{eq rcb r1}. We believe to the truth of this assertion, but our main difficulty in carrying out its proof is the use of Proposition \ref{prop crca}, where the weight of non-constant elements of the graded algebra are considered positive. This led us to the following conjecture.

\begin{conj}\label{conj 1}
The proposition \ref{prop crca} holds if the graded algebra $M$, besides elements of positive weight, also contains elements of negative weight or non-constant elements of weight zero. In the other words, if $M=\bigoplus_{k\in \Z}M_k$, in which it is not necessary that $M_0=\k.1$.
\end{conj}
In the above conjecture by constant elements we mean the elements of the field $\k$. If we want to prove Conjecture \ref{conj 1} in an analogous way to the proof of Zagier given for \cite[Proposition 1]{zag94}, the unsolved part is the equality \eqref{eq []=[]}. Once we prove Conjecture \ref{conj 1}, we can prove that the space of \cy modular forms $\cymfs$ is closed under the Rankin-Cohen brackets \eqref{eq rcb r1}.

Since the \cy 3-folds are more important in the literature, we state the Gauss-Manin connection matrix of $\rcv$ for $n=3$ here:
\begin{equation}\label{eq GM_D_3}
\gm_\rcv=\left(
  \begin{array}{cccc}
    0 & 1 & 0 & 0 \\
    0 & 0 & \Yuk_1 & 0 \\
    t_2t_4 & 0 & 0 & -1 \\
    -t_2(t_2t_4+t_7) & t_2t_4 & 0 & 0 \\
  \end{array}
\right)\,,
\end{equation}
in which $\Yuk_1=\frac{ t_3^3}{5^4(t_1^5-t_5)}$. Note that, due to Theorem \ref{main3}, the Gauss-Manin connection matrix of $\Ra$ is as follows:
\begin{equation}\label{eq GM_R_3}
\gm_\Ra=\left(
  \begin{array}{cccc}
    0 & 1 & 0 & 0 \\
    0 & 0 & \Yuk_1 & 0 \\
    0 & 0 & 0 & -1 \\
    0 & 0 & 0 & 0 \\
  \end{array}
\right)\,.
\end{equation}
in which we also have $\Yuk_1=\frac{ t_3^3}{5^4(t_1^5-t_5)}$. If we substitute the solutions of $\Ra$ in $\Yuk_1$, then we get the Yukawa coupling.
It would be very interesting, and maybe helpful, if one can find out the (physical) interpretation of the non-zero part of the lower triangle of the matrix $\gm_\rcv$ stated in \eqref{eq GM_D_3}.


\def\cprime{$'$} \def\cprime{$'$} \def\cprime{$'$}





\end{document}